\documentclass[11pt,letterpaper]{amsart}
\usepackage{euscript,graphicx,amscd,amsgen,amsfonts,amssymb,latexsym,
amsmath,amsthm,graphicx,mathrsfs,times,color,overpic,bm}
\usepackage[all]{xy}
\usepackage{enumitem}

\usepackage[T1]{fontenc}

\newtheorem{theorem}{Theorem}[section]
\newtheorem{lemma}[theorem]{Lemma}
\newtheorem{claim}[theorem]{Claim}
\newtheorem*{claim*}{Claim}
\newtheorem{corollary}[theorem]{Corollary}
\newtheorem{proposition}[theorem]{Proposition}
\newtheorem{assumption}[theorem]{Assumption}

\newtheorem{mtheo}{Theorem}

\theoremstyle{definition}
\newtheorem{definition}[theorem]{Definition}

\newtheorem{remark}[theorem]{Remark}
\newtheorem*{maintheorem*}{Theorem}

\newcommand{\eqdef}{\stackrel{\scriptscriptstyle\rm def}{=}}
\newcommand{\spac}[1]{{\quad\text{ #1 }\quad}}

\DeclareMathOperator{\var}{var}

\DeclareMathOperator{\diam}{diam}

\DeclareMathOperator{\card}{card}
\DeclareMathOperator{\interior}{int}
\DeclareMathOperator{\Lip}{Lip}

\def\bp{\mathfrak p}

\def\Sub{\mathcal{S}}

\def\cW{\EuScript{B}}
\def\cS{\EuScript{S}}

\def\frakR{\mathfrak{R}}

\def\fS{\EuScript{S}}
\def\fT{\EuScript{T}}
\def\fX{\EuScript{X}}

\def\cC{\mathcal{C}} 

\def\fX{\mathfrak{X}}
\def\fY{\mathfrak{Y}}
\def\fE{\mathfrak{E}}
\def\fH{\mathfrak{H}}
\def\fm{\mathfrak{m}}

\def\fT{\mathfrak{T}}

\def\bnu{\boldsymbol{\nu}}

\def\ua{\underline{a}}

\def\bt{\mathbf{t}}

\def\fm{\mathfrak m}
\def\ft{\mathfrak t}

\def\bA{\mathbb{A}}
\def\bP{\mathbb{P}}
\def\bE{\mathbb{E}}
\def\bN{\mathbb{N}}

\def\bZ{\mathbb{Z}}
\def\bR{\mathbb{R}}
\def\bS{\mathbb{S}}

\def\cA{\EuScript{A}}
\def\cB{\EuScript{B}}

\def\cN{\EuScript{N}}
\def\cO{\EuScript{O}}
\def\cP{\EuScript{P}}

\def\cV{\EuScript{V}}
\def\cN{\EuScript{N}}

\def\cW{\mathcal{W}}

\def\cM{\EuScript{M}}

\numberwithin{equation}{section}

\DeclareMathSymbol{\varnothing}{\mathord}{AMSb}{"3F}
\renewcommand{\emptyset}{\varnothing}

\title[Loosely Bernoulli zero exponent measures]{Loosely Bernoulli zero exponent measures\\for elliptic matrix cocycles}
\author[ L.~J.~D\'iaz]{L. J. D\'\i az}
\address{Departamento de Matem\'atica PUC-Rio, Marqu\^es de S\~ao Vicente 225, G\'avea, Rio de Janeiro 22451-900, Brazil}
\email{lodiaz@mat.puc-rio.br}
\author[K.~Gelfert]{K. Gelfert}
\address{Instituto de Matem\'atica Universidade Federal do Rio de Janeiro, Av. Athos da Silveira Ramos 149, Cidade Universit\'aria - Ilha do Fund\~ao, Rio de Janeiro 21945-909,  Brazil}
\email{gelfert@im.ufrj.br}
\author[M.~Rams]{M. Rams} \address{Institute of Mathematics, Polish Academy of Sciences, ul. \'{S}niadeckich 8,  00-656 Warszawa, Poland}
\email{rams@impan.pl}
\thanks{We would like to express our deep gratitude to Dominik Kwietniak for his patient and lucid answers to our multiple questions about the ergodic theory. 
Most part of this research was done while the authors participated in the Thematic Research Program `Modern holomorphic dynamics and related fields', Excellence Initiative -- Research University program at the University of Warsaw, 2023. This research has been supported [in part] by 
CAPES -- Finance Code 001, by 
CNPq-grants  310069/2020-3 
430154/2018-6,  
305327/2022-4, and 
E-16/2014 INCT/FAPERJ, 
E-26/211.313/2021 FAPERJ, 
E-26/200.371/2023 CNE/FAPERJ, 
and
PRONEX E-26/010.001252/2016 FAPERJ (Brazil).  
MR was also partially supported by National Science Centre grant 2019/33/B/ST1/00275 (Poland). 
The authors thank their home institutions for the hospitality during their visits while preparing this paper.
}
\begin{document}

\begin{abstract}
For an open and dense subset of elliptic ${\rm SL}(2,\bR)$ matrix cocycles, we construct a family of loosely Bernoulli ergodic measures with zero top Lyapunov exponent. This provides a counterpart to a classical result by Furstenberg. The construction gives also an $\bar f$-connected set of measures with these properties whose entropies vary continuously from zero to almost the maximal possible value. We also obtain an analogous result for an open class of nonhyperbolic step skew products with $\bS^1$ diffeomorphism fiber maps. Our approach combines substitution schemes between finite letter alphabets and differentiable dynamics.
\end{abstract}

\maketitle


\section{Introduction}

We study matrix cocycles generated by finite families $\bA\eqdef\{A_1,\ldots,A_N\}$, $N\ge 2$, of
${\rm SL}(2,\bR)$ (the space of $2\times 2$ matrices with real coefficients and determinant one)
and consider the \emph{top Lyapunov exponent}
\[
	\lambda_1(\bA,\nu^+)
	\eqdef\lim_{n\to\infty}\frac1n\int\log\,\lVert \bA^n(\xi^+)\rVert\,d\nu^+(\xi^+),\quad
		\bA^n(\xi^+)\eqdef \lVert A_{\xi_{n-1}}\circ\cdots\circ A_{\xi_0}\rVert,
\] 
where $\xi^+=(\xi_0,\xi_1,\ldots)\in\Sigma_N^+\eqdef\{1,\ldots,N\}^{\bN_0}$ and $\nu^+$ is a shift-invariant ergodic probability measure on $\Sigma_N^+$. Its existence is now a standard consequence of Kingman's subadditive ergodic theorem, but was first established in  \cite{FurKes:60}, who showed that any stationary stochastic process taking values in a set of matrices has asymptotically an exponential growth rate. 
Moreover, by Furstenberg \cite{Fur:63}, excluding some well-defined exceptional cases, for all identically and independently distributed (i.i.d.) processes the top Lyapunov exponent is positive. Furstenberg's result was generalized in \cite{Vir:79} to stationary Markov-dependent sequences and in \cite{Gol:22} to the non-stationary Markov case.%
\footnote{All those results are, in fact, stated in much higher generality, but in this paper we only deal with the two-dimensional case.} 
See also \cite{BarMal:20} for a version of random iterations of conservative diffeomorphisms.

To state a specific, yet pertinent, case recall that a cocycle $\bA=\{A_1,\ldots,A_N\}$ is \emph{strongly irreducible} if there does not exist any finite collection $V_1,\ldots,V_m$ of nonzero proper subspaces so that $A_i(\bigcup_{j=1}^mV_j)=\bigcup_{j=1}^mV_j$ for every $i=1,\ldots,N$. Moreover, $\bA$ is \emph{proximal} if there is $(i_1,\ldots,i_n)\in\{1,\ldots,N\}^n$ such that the matrix resulting from the composition $A_{i_n}\circ\cdots\circ A_{i_1}$ has two real eigenvalues with different absolute values.

\begin{maintheorem*}[{Furstenberg \cite{Fur:63}}]
	Assume that $\bA\subset{\rm SL}(2,\bR)^N$ is strongly irreducible and proximal. Then for every nondegenerate Bernoulli measure $\nu^+$ it holds $\lambda_1(\bA,\nu^+)>0$.
\end{maintheorem*}

The necessity of both hypotheses  to guarantee the positivity of the top Lyapunov exponent can be seen from the following examples: given $\bA=\{A_1,A_2\}$ so that  
\begin{itemize}[leftmargin=0.5cm ]
\item either $A_1$ and $A_2$ are both rotations (hence $\bA$ is not proximal); 
\item or $A_1=\left(\begin{matrix}\lambda&0\\0&\lambda^{-1}\end{matrix}\right)$,  for some $\lambda>1$,  and $A_2=\left(\begin{matrix}0&-1\\1&0\end{matrix}\right)$ (hence $\bA$ is not strongly irreducible),
\end{itemize}
then every non-degenerate Bernoulli measure $\nu^+$ satisfies $\lambda_1(\bA,\nu^+)=0$.

We prove the existence, for an open and dense subset of elliptic ${\rm SL}(2,\bR)$ matrix cocycles which satisfy the assumptions of Furstenberg's theorem, of a loosely Bernoulli measure with zero top Lyapunov exponent.
This way, we limit how far Furstenberg-like results can be generalized. In some sense, loosely Bernoulli measures have ``short-range memory'' and as such are ``not too far away'' from an i.i.d. process. Historically, the loosely Bernoulli property was introduced by Weiss and further developed by Feldman \cite{Fel:76}, Katok \cite{Kat:77}, and Ornstein, Rudolph, and Weiss \cite{OrnRudWei:82}.
In very rough terms, a positive entropy automorphism is loosely Bernoulli if it is a discrete flow that has a measurable cross-section so that its first return is a Bernoulli automorphism. We postpone the definition and their properties to Section \ref{secLB}. 

Let us introduce our setting more precisely. The space $\mathrm{SL}(2, \mathbb{R})^N$ roughly splits into the subsets of \emph{elliptic} and \emph{uniformly hyperbolic} cocycles (denoted by $\fE_N$ and $\fH_N$, respectively).  Both sets  $\fE_N$ and $\fH_N$ are open and their union is dense in $\mathrm{SL}(2, \mathbb{R})^N$, see \cite[Proposition 6]{Yoc:04}. Hyperbolic cocycles are quite well understood and have positive top Lyapunov exponent (regardless of the base measure), \cite{AviBocYoc:10}. In this way, when searching for zero top Lyapunov exponents, we need to focus on elliptic cocycles, that are not so well understood. Recall that $\bA$ is \emph{elliptic} if its associated multiplicative semigroup contains some \emph{elliptic} element (that is, the absolute value of its trace is less than $2$). Here we consider the subset of {\em{elliptic cocycles with some hyperbolicity}} $\fE_{N, \rm shyp}$ of $\fE_N$ introduced in \cite{DiaGelRam:19}. The set $\mathfrak{E}_{N,\rm shyp}$ is open and dense in $\mathfrak{E}_N$. Moreover, every $\bA\in \fE_{N, \rm shyp}$ is strongly irreducible and proximal, and therefore Furstenberg's Theorem applies.  

 As indicated by the variational principle \cite[Theorem B]{DiaGelRam:22a}, for every cocycle $\bA  \in \fE_{N,\rm shyp}$ there are plenty of ergodic measures $\nu^+$ whose top Lyapunov exponent is zero. See also \cite{BocRam:16,Fen:09} for results illustrating this fact. Furstenberg's theorem implies that those measures cannot be Bernoulli (and not even Markov, by \cite{Vir:79,Gol:22}).
By Theorem~\ref{theMain} below, some of those measures are loosely Bernoulli. 

We also ask what is the ``maximal complexity'' of measures with exponent zero? Denote by $\sigma^+$  the usual left shift on $\Sigma_N^+$, $\cM_{\rm erg}(\Sigma_N^+,\sigma^+)$ is the set of ergodic probability measures, and $h(\sigma^+,\nu^+)$ is the metric entropy of a measure $\nu^+$ (with respect to $\sigma^+$). Using entropy as a quantifier, let us introduce
\begin{equation}\label{h0A}\begin{split}
	&h_0(\bA)
	\eqdef \sup\big\{h(\sigma^+,\nu^+)\colon\nu^+\in\cM_{\rm erg,0}(\Sigma_N^+,\sigma^+)\big\},
	\quad\text{where } \\
	&
	\cM_{\rm erg,0}(\Sigma_N^+,\sigma^+)
	\eqdef\{\nu^+\in \cM_{\rm erg}(\Sigma_N^+,\sigma^+)\colon \lambda_1(\bA,\nu^+)=0\}.
\end{split}\end{equation}
As shown in \cite{DiaGelRam:22a}, we have $h_0(\bA)\in(0,\log N)$. By Theorem~\ref{theMain} below, loosely Bernoulli zero exponent measures can have entropy arbitrarily close to $h_0(\bA)$. 

Our study also relates to ergodic optimization theory, describing the nature of the measures giving rise to the extremal values of the top Lyapunov exponent. While the maximal exponent measures typically are dynamically simple (periodic or of zero entropy, see the survey \cite{Jen:19} for more information),  the minimal exponent measures are often large and dynamically complicated (see for example \cite{BocRam:16,DiaGelRam:22a}).
  
The following is our main result in the setting of matrix cocycles. Equip the space $\Sigma_N^+$ with the $\bar f$-distance. We will provide more details on the $\bar f$-topology and loosely Bernoulli automorphisms in Section \ref{secLB}. For now, just note that $\bar f$-convergence implies convergence in the weak$\ast$ topology and in entropy.

\begin{mtheo}\label{theMain}
	For every $N\ge2$, there is an open and dense subset $\mathfrak{E}_{N,\rm shyp}$ of $\mathfrak{E}_N$ such that for every $\bA\in\mathfrak{E}_{N,\rm shyp}$ and every $\varepsilon\in (0,h_0(\bA))$, there is a $\bar f$-path-connected set $\cN_\varepsilon\subset\cM_{\rm erg,0}(\Sigma_N^+,\sigma^+)$ such that
\begin{enumerate}
\item each $\nu^+\in\cN_\varepsilon$ is loosely Bernoulli,
\item the entropy $h(\sigma^+,\nu^+)$ varies $\bar f$-continuously in $\nu^+\in\cN_\varepsilon$ and 
\[
	\big\{h(\sigma^+,\nu^+)\colon\nu^+\in\cN_\varepsilon\big\}
	\supset\big[0,h_0(\bA)-\varepsilon\big].
\]
\end{enumerate}	
\end{mtheo}

We will use the term {\it weak* and entropy-convergence}, that is, simultaneously weak* convergence and convergence in entropy. This notion plays also an important role in multifractal analysis and large deviation theory.

\begin{mtheo}\label{theMainb}
	Under the hypotheses of Theorem \ref{theMain}, the set of measures
\[
	\{\nu^+\in\cM_{\rm erg,0}(\Sigma_N^+,\sigma^+)\colon \nu^+\text{ is loosely Bernoulli}\}
\]	
is weak$\ast$ and entropy-dense in $\cM_{\rm erg,0}(\Sigma_N^+,\sigma^+)$.
\end{mtheo}

As our main tool, to prove the above theorems, we study step skew products associated to $C^1$ circle diffeomorphisms $f_1,\ldots,f_N\colon\bS^1\to\bS^1$ over the two-sided shift $\sigma$ defined on $\Sigma_N\eqdef\{1,\ldots,N\}^\bZ$ and their fiber Lyapunov exponents. Writing $\xi=(\ldots,\xi_{-1}|\xi_0,\xi_1,\ldots)$, let
\begin{equation}\label{eq:sp}
	F\colon \Sigma_N\times \bS^1\to \Sigma_N\times \bS^1,
	\quad
	F(\xi,x) 
	\eqdef (\sigma(\xi), f_{\xi_0}(x)).
\end{equation}
In our study, we combine differentiable methods that we apply on $\Sigma_N\times\bS^1$ with certain substitution schemes used directly on the symbolic space $\Sigma_N$. Our main object of study is the class of 
step skew products $\mathrm{SP}^1_{\rm shyp}(\Sigma_N\times\bP^1)$ introduced in \cite{DiaGelRam:17}. They are roughly characterized by the existence of an ``expanding'' and a ``contracting region'' (relative to the fiber direction) that are intermingled by the dynamics, we postpone the details till Section \ref{seccocyclediffeo}. These skew products are robustly transitive and robustly nonhyperbolic, in the sense that the existence of zero fiber exponent measures is not removable after perturbation, and exhibit ergodic measures with zero fiber Lyapunov exponent with positive entropy.  

Let us now present our two main theorems. Theorems \ref{theMain} and \ref{theMainb} are their almost immediate consequences. Consider the \emph{fiber Lyapunov exponent} of an $F$-invariant measure $\eta$ defined by
\begin{equation}\label{h0Fdef}
		\chi(F,\eta)
	\eqdef \int\log\,\lvert f_{\xi_0}'(x)\rvert\,d\eta(\xi,x).
\end{equation}
Analogously to \eqref{h0A}, and with a slight abuse of notation, let
\begin{equation}\label{h0F}\begin{split}
	&h_0(F)
	\eqdef \sup\big\{h(F,\eta)\colon\eta\in\cM_{\rm erg,0}(\Sigma_N\times\bS^1,F)\big\},\,\text{ where }\\
	&\cM_{\rm erg,0}(\Sigma_N\times\bS^1,F)
	\eqdef\big\{\mu\in \cM_{\rm erg}(\Sigma_N\times\bS^1,F)\colon\chi(F,\mu)=0\big\}.
\end{split}\end{equation}
As shown in \cite{DiaGelRam:22a}, we have $h_0(F)>0$ for every $F\in \mathrm{SP}^1_{\rm shyp}(\Sigma_N\times\bS^1)$. Let $\pi\colon\Sigma_N\times\bS^1\to\Sigma_N$ be the canonical projection to the first coordinate. Note that $\pi_\ast$ preserves the metric entropy (see \cite{LedWal:77}).

\begin{mtheo} \label{Bthm:circle}
	For every $F\in\mathrm{SP}^1_{\rm shyp}(\Sigma_N\times\bS^1)$, $N\ge2$, and every $\varepsilon>0$ there is a weak$\ast$ path-connected set $\cM_\varepsilon \subset\cM_{\rm erg,0}(\Sigma_N\times\bS^1,F)$ such that 
\begin{enumerate}
\item the projection  $\pi_\ast(\cM_\varepsilon)$ of $\cM_\varepsilon$ to $\cM_{\rm erg}(\Sigma_N,\sigma)$ is $\bar f$-path connected and consists of loosely Bernoulli measures,
\item  the entropy $h(F,\mu)$ varies weak$\ast$-continuously in $\mu\in\cM_\varepsilon$, the entropy $h(\sigma, \nu)$ varies $\bar f$-continuously in $\nu\in\pi_\ast(\cM_\varepsilon)$, and
\[
	\big\{h(F,\mu)\colon\mu\in\cM_\varepsilon \big\}\supset [0,h_0(F) -\varepsilon].
\]
\end{enumerate}
\end{mtheo}

\begin{mtheo} \label{Bthm:circleb}
	Under the hypotheses of Theorem \ref{Bthm:circle}, the set of measures
\[
	\big\{\mu\in\cM_{\rm erg,0}(\Sigma_N\times\bS^1,F),\pi_\ast\mu\text{ is loosely Bernoulli}\big\}
\]	
is weak$\ast$ and entropy-dense in $\cM_{\rm erg,0}(\Sigma_N\times\bS^1,F)$.
\end{mtheo}

Our results, besides showing the abundance of loosely Bernoulli measures with zero exponent, also shed some light on the general structure of the space of ergodic measures. This direction of research follows the line of \cite{Sig:74,Sig:77} for systems with specification.  More recent contributions beyond specification in this very active field are, for example, \cite{GorPes:17} (path connectedness and the Poulsen property for ergodic measures in homoclinic classes), \cite{GelKwi:18} (density of ergodic measures assuming closability and linkability), \cite{DiaGelRam:17} (weak$\ast$ and in entropy-approximation of zero exponent measures for maps in $\mathrm{SP}^1_{\rm shyp}(\Sigma_N\times\bP^1)$), \cite{DiaGelSan:20,YanZha:20} (weak$\ast$ and in entropy-approximation of zero exponent measures for fairly general partially hyperbolic diffeomorphisms), and  \cite{BonZha:19} (density of periodic measures in partially hyperbolic homoclinic classes).

A scheme to construct nonhyperbolic (that is, with zero fiber Lyapunov exponent) 
ergodic measures for skew products of circle diffeomorphism was introduced in \cite{GorIlyKleNal:05} based on a method of period orbit approximation. In \cite{KwiLac:}, the loosely Kronecker property is shown for measures in this construction (that is, they are loosely Bernoulli and have zero entropy). The construction in  \cite{DiaGelRam:22a} replaces the periodic orbits used in  \cite{GorIlyKleNal:05} by horseshoes, in order to get zero exponent ergodic measure with entropy that is as large as possible. This approach is via cascades of contracting iterated function systems as well as ``horseshoes'' in $\Sigma_N\times\bS^1$. The focus in  \cite{DiaGelRam:22a} is put only on entropy and ergodicity. 

Here we provide a conceptually novel description of the method in \cite{DiaGelRam:22a}. One of our key tools are \emph{Bernoulli-coded measures} that are, by definition, images of Bernoulli measures under substitution maps between symbolic spaces with finite alphabets. We do a probabilistic investigation of Bernoulli-coded measures on $\Sigma_N$. In this way, we enlarge the class of measures constructed following \cite{DiaGelRam:22a}. Moreover, this approach enables a finer analysis of their inner structure and ergodic properties. 

Let us finally describe the relation between Theorems \ref{theMain} and \ref{Bthm:circle}. Given a family $\bA=\{A_1,\ldots,A_N\}\in\mathrm{SL}(2, \mathbb{R})^N$, $N\ge2$, the action of any matrix on the projective line $\bP^1$ (which is topologically the circle $\bS^1$) provides very special diffeomorphisms: for $i=1,\ldots,N$ let
\begin{equation}\label{neq:defsteskecoc}
	f_i\colon \bP^1 \to \bP^1, \quad
	f_i (v)
	\eqdef \frac{A_iv}{\lVert A_i v\rVert}.
\end{equation}
We focus on the class of cocycles $\fE_{N, \rm shyp}$ mentioned above. A key property is that any $\bA$ in this class gives rise to a step skew product in $\mathrm{SP}^1_{\rm shyp}(\Sigma_N\times\bP^1)$. Moreover, the fiber Lyapunov exponent $\chi(F,\cdot)$ of this skew product is related to the top Lyapunov exponent $\lambda_1(\bA,\pi_\ast(\cdot))$. Theorem \ref{theMain} will then be an almost immediate consequence of Theorem \ref{Bthm:circle}. See Section \ref{ImpMatCoc} for details.

The paper is organized as follows. We start by defining the loosely Bernoulli property and the $\bar f$-topology in Section \ref{secLB} and by recalling their main properties. In Section \ref{secBerCod}, we introduce the concept of substitutions between symbolic spaces with finite alphabets. In particular, we investigate, in the $\bar f$-topology, the class of Bernoulli-coded measures that are defined as images of Bernoulli measures under such substitutions.  In Section \ref{secabscasc}, we study a cascade of substitutions into a common alphabet. This section provides the symbolic half of our main arguments. The other, geometric, half of our arguments is provided in Sections \ref{secGeometry} and \ref{secGeometry-2}. In Section \ref{secGeometry}, we recall and explain in details the class of maps $\mathrm{SP}^1_{\rm shyp}(\Sigma_N\times\bS^1)$ and we describe the structure that stands behind our construction of zero fiber Lyapunov exponent ergodic measures. In Section \ref{secGeometry-2}, we analyze the properties of the limit measures, in particular prove their ergodicity. Section \ref{sec7} we prove Theorems \ref{Bthm:circle} and \ref{Bthm:circleb}. In Section \ref{ImpMatCoc}, we prove Theorems \ref{theMain} and \ref{theMainb}.  

\section{Loosely Bernoulli automorphisms}\label{secLB}

In this section, we define the {\em{loosely Bernoulli}} property  and recall some essential properties of loosely Bernoulli automorphisms. For that, we will make use of the concept of $\bar f$-distance. 
The $\bar f$-distance between strings of symbols was introduced by Feldman \cite{Fel:76}, replacing the Hamming metric in the definition of Ornstein's very weak Bernoulli property with the edit distance.
In the case of zero entropy systems, the loosely Bernoulli property was introduced independently by Katok \cite{Kat:75}.
The results in this section are collected from \cite{OrnRudWei:82}.

Let $\cA$  be a finite alphabet.
A \emph{word} over $\cA$ is a finite sequence of symbols in $\cA$. The \emph{length} of a word $a$ is the number of symbols it contains and is denoted by $|a|_\cA$ or simply by $|a|$ if $\cA$ is clear from the context. The \emph{empty word} is the unique word with no symbols and has length zero. An \emph{$n$-word} is a word of length $n$. The set of all words over the alphabet $\cA$ (including the empty one) is denoted by $\cA^\ast$. A \emph{substring} of a word $(a_1,\ldots,a_n)$ is any word of the form $(a_{i_1},\ldots,a_{i_k})$ for some numbers $1\le i_1<i_2<\cdots<i_k\le n$. We also consider the space $\cA^\bZ$ of all bi-infinite sequences of symbols from $\cA$. The elements in 
$\cA^\bZ$ are denoted by $\underline a=(\ldots,a_{-1}|a_0,a_1,\ldots)$.

\begin{definition}[Distance on word space]
The \emph{edit distance} (of level $n$) of two $n$-words $a=(a_1,\ldots,a_n)$ and $b=(b_1,\ldots,b_n)$,  is
\[
	\bar{f}_n(a,b)
	\eqdef 1-\frac kn,
\]
where $k$ is the largest number such that for some indices $1\le i_1<i_2<\cdots<i_k\le n$ and $1\le j_1<j_2<\cdots<j_k\le n$ it holds $a_{i_s}=b_{j_s}$ for every $s=1,\ldots,k$. In other words, the edit distance between two $n$-words is given in terms of the relative length of their maximal common substring.
\end{definition}

\begin{definition}[$\bar f$-distance on the sequence space]
For two bi-infinite sequences $\underline a=(\ldots,a_{-1}|a_0,a_1,\ldots),\underline b=(\ldots, b_{-1}|b_0,b_1,\ldots)\in\cA^{\bZ}$, define the \emph{Feldman pseudometric} or \emph{$\bar f$-pseudometric} on $\cA^{\bZ}$ by
\[
	\bar{f}(\underline a,\underline b)
	\eqdef \limsup_{n\to\infty}\bar{f}_{2n}\big((a_{-n},\ldots,a_{n-1}),(b_{-n},\ldots,b_{n-1})\big).
\]
\end{definition}

Let us define now also the $\bar{f}$-pseudometric between measures. For notational simplicity, we use the same symbol.

Equip the space $\cA^\bZ$ with the metric $d_1(\underline a,\underline b)\eqdef e^{-\inf\{\lvert k\rvert\colon a_k\ne b_k\}}$. Note that $(\cA^\bZ,d_1)$ is a compact metric space whose topology coincides with the one generated by cylinder sets. Consider the Borel $\sigma$-algebra on $\cA^\bZ$.

\begin{definition}[$\bar f$-distance between measures]
Given two probability measures $\nu,\nu'$ on $\cA^\bZ$ and $n\in\bN$, denote by $\nu_n,\nu_n'$ their restrictions to the $\sigma$-field generated by all $n$-cylinders $\cA^n$. Denote by $J_n(\nu,\nu')$ the set of all (probability) measures on $\cA^n\times\cA^n$ whose marginals are $\nu_n,\nu_n'$, respectively. We refer to the elements of $J_n(\nu,\nu')$ shortly as the \emph{$n$-joinings} of $\nu$ and $\nu'$. Let
\begin{equation}\label{eqfbar}
	\bar{f}_n(\nu,\nu')
	\eqdef \inf_{\bnu_n\in J_n(\nu,\nu')}\int_{\cA^n\times\cA^n}\bar{f}_n\,d\bnu_n.
\end{equation}
Define by
\[
	\bar{f}(\nu,\nu')
	\eqdef \inf\big\{\varepsilon\colon \bar f_n(\nu,\nu')\le\varepsilon\text{ for infinitely many }n\big\}
\]
the {\em{\emph{$\bar f$}-distance}} of $\nu$ and $\nu'$.
\end{definition}

Consider the shift map $\sigma_\cA\colon\cA^\bZ\to\cA^\bZ$ defined by $(\sigma_\cA(\underline a))_k=a_{k+1}$, for all $k\in\bZ$. Denote by $\cM_{\rm erg}(\cA^\bZ,\sigma_\cA)$ the space of all Borel probability measures which are ergodic with respect to $\sigma_\cA$.

\begin{remark}\label{remfbars}
The $\bar f$-distance is a metric on the space $\cM_{\rm erg}(\cA^\bZ,\sigma_\cA)$ of ergodic measures. The topology induced by this metric is complete and stronger than the weak$\ast$ topology. Moreover, $\bar f$-convergence implies convergence in entropy. See \cite[Section 2]{OrnRudWei:82} for further details.
\end{remark}

Indeed, for further reference, let us recall that the entropy map $\cM_{\rm erg}(\cA^\bZ,\sigma_\cA)\ni\nu\mapsto h(\sigma_\cA,\nu)$ is even uniformly continuous (see \cite[Proposition 3.4]{OrnRudWei:82}).

\begin{lemma}\label{lementcont}
	For any measures $\nu_1,\nu_2\in\cM_{\rm erg}(\cA^\bZ,\sigma_\cA)$ satisfying $\bar f(\nu_1,\nu_2)<\varepsilon$, 
\[
	|h(\sigma_\cA,\nu_1)-h(\sigma_\cA,\nu_2)|
	\le 2\big(-\varepsilon\log \varepsilon-(1-\varepsilon)\log(1-\varepsilon)\big)+\varepsilon\log \card\cA.
\]
\end{lemma}

\begin{definition}[Loosely Bernoulli property (LB)]
Given $(X,\fX,S,\mu)$ an automorphism, a finite partition $\cP$ of $X$ is \emph{loosely Bernoulli (LB)} if for every $\varepsilon>0$ there are $n\in\bN$ and a set $A_n$ of atoms of $\cP^n\eqdef\bigvee_{j=0}^{n-1}S^{-j}(\cP)$ such that $\mu(A_n)>1-\varepsilon$ and $\bar{f}_n(v,w)<\varepsilon$ for every $v,w\in A_n$ (here, we identify the atoms of a partition $\cP^n$ with words of length $n$ over the alphabet $\{1,\ldots,|\cP|\}$.
We say that $(X,\fX,S,\mu)$ is \emph{LB} if every finite partition is LB. 
\end{definition}

\begin{remark}
	An automorphism $(X,\fX,S,\mu)$ is LB if, and only if, it is Kakutani equivalent	 to a Bernoulli automorphism.
	 For that recall that two measure preserving systems $(X,\fX,S,\mu)$ and $(Y,\fY,T,\nu)$ are \emph{Kakutani equivalent} if there exist $A\in\fX$ with $\mu(A)>0$ and $B\in\fY$ with $\nu(B)>0$ such that the induced transformations $S_A$ and $T_B$ are isomorphic.%
\footnote{Given a measure preserving system $(X,\fX,S,\mu)$ and $A\in\fX$ with $\mu(A)>0$, for $x\in A$ we define its \emph{first return time} $n(x)=\inf\{n\in\bN\colon S^n(x)\in B\}$. By Birkhoff ergodic theorem, this function is finite for almost every $x\in A$. The associated \emph{induced transformation} $S_A\colon A\to A$, $S_A(x)=   S^{n(x)}(x)$ is well defined almost everywhere.}
	Note that, for $(X,\fX,S,\mu)$ to be LB it suffices to be LB for some generating partition (compare \cite[Theorem 4.6]{OrnRudWei:82}).
\end{remark}

\begin{remark}[General properties of LB automorphisms]\label{remarkLB}
Every Bernoulli shift is LB. By \cite[Theorem 3]{Fel:76}, every factor of a LB automorphism is LB. By \cite[Corollary 4.9]{OrnRudWei:82}, every Markov automorphisms is LB. Note that every LB automorphism is ergodic.

Given an automorphism $(X,\fX,S,\mu)$ and $A\in\fX$ satisfying $\mu(A)>0$,  then $S$ is LB if and only if the induced transformation $S_A$ is LB. In particular, the LB property is an invariant of Kakutani equivalence.
\end{remark}

Below, we always consider the Borel $\sigma$-algebra and omit it in our notations. 

\begin{lemma}\label{lemLB}
	Let $(\nu_n)_n\subset\cM_{\rm erg}(\cA^\bZ,\sigma_\cA)$ be a sequence of measures such that $(\cA^\bZ,\sigma_\cA,\nu_n)$ if LB for every $n$. If this sequence $\bar f$-converges to some $\nu$, then $(\cA^\bZ,\sigma_\cA,\nu)$ is LB.
\end{lemma}

\begin{proof}
This seems to be a folklore result based on the results from \cite{OrnRudWei:82}. Indeed, recall the concepts \emph{finitely fixed (FF)} and LB (see \cite[Section 4]{OrnRudWei:82}), which are equivalent by \cite[Theorems 6.5 and 6.7]{OrnRudWei:82}. Then the assertion is a consequence of \cite[Theorem 4.7]{OrnRudWei:82}.
\end{proof}

\section{Bernoulli-coded measures}\label{secBerCod}

In this section, we introduce substitution maps between sequence spaces over finite alphabets and associated Bernoulli-coded measures. In order to study some fundamental properties, in particular the $\bar f$-distance between them, we investigate suspension spaces. This is done in Section \ref{secfbardistance}. 

\subsection{Preliminaries}

Let $\cA$  be a finite alphabet. 
The \emph{concatenation} of words $a=(a_1,\ldots,a_k)$ and $b=(b_1,\ldots,b_\ell)$ is the word $ab\eqdef(a_1,\ldots,a_k,b_1,\ldots,b_\ell)$. 
The concatenation of any number of words, including infinite and bi-infinite, is analogously defined. A \emph{prefix} of a word $b\in\cA^\ast$ is a word $a\in\cA^\ast$ such that $b=ac$ for some $c\in\cA^\ast$.  A \emph{suffix} of a word $b\in\cA^\ast$ is a word $a\in\cA^\ast$ such that $b=ca$ for some $c\in\cA^\ast$. A \emph{subword} of a word $b\in\cA^\ast$ is a word $a\in\cA^\ast$ such that $b=cad$ for some $c,d\in\cA^\ast$. Note that the notation for concatenation of words does not use commas. 

Given a word $a=(a_0,\ldots,a_{k-1})\in\cA^\ast$, we use the standard notation
\begin{equation}\label{defcylinder}
[a]\eqdef \{\underline b=(\ldots,b_{-1}|b_0,b_1,\ldots)\colon b_i=a_i\text{ for }i=0,\ldots,k-1\}
\end{equation}
to denote the corresponding cylinder set. We only consider cylinders starting at position ``$0$''.  

\begin{definition}[Space of Bernoulli measures]
Consider a probability vector $\bp=(p_a)_{a\in \cA}$, $p_a\ge0$ and $\sum_{a\in \cA}p_a=1$. To simplify our notation, we denote by $\bp$ also the Bernoulli measure on $\cA^\bZ$ defined by this vector.  Let
\begin{equation}\label{defBernoullis}
	\cM_{\rm B}(\cA^\bZ,\sigma_\cA)
	\eqdef \big\{\bp\in\cM_{\rm erg}(\cA^\bZ,\sigma_\cA)\colon \bp\text{ is Bernoulli}\big\}
\end{equation}
and consider the \emph{city metric} on this space:
\[
	D_\cA(\bp,\bp')
	\eqdef \sum_{a\in\cA}|p_a-p_a'|.
\]
\end{definition}

Let us state some auxiliary result.

\begin{lemma}\label{lemshields}
	 For every $\bp,\bp'\in\cM_{\rm B}(\cA^\bZ,\sigma_\cA)$,
$
	\bar{f} (\bp,\bp')
=  \frac12D_\cA(\bp,\bp') .
$
\end{lemma}

\begin{proof}
	Note that the $\bar f$-distance is bounded from above by the $\bar d$-distance and
\[
	\bar f(\bp,\bp')
	\le \bar d(\bp,\bp')
	= \frac12D_\cA(\bp,\bp') ,
\]
this equality and the definition of the $\bar d$-distance can be found in  \cite[Chapter I.9]{Shi:96}. Hence, it only remains to show the lower bound. For that observe that, for $n$ large enough, $\bp$-almost every $n$-word contains approximately $np_a$ symbols $a$, for any $a\in\cA$. Analogously for $\bp'$. Hence, the edit distance (at level $n$) is bounded from below by
\begin{equation}\label{llnw}
	1-\frac1n\sum_{a\in\cA}\min\{np_a,np_a'\}
	= \sum_{a\in\cA}\frac12(p_a+p_a')-\min_{a\in\cA}\{p_a,p_a'\}
	= \frac12D_\cA(\bp,\bp').
\end{equation}
Note this implies that, for any $n$-joining of $\bp$ and $\bp'$, the integrand in \eqref{eqfbar} is bounded from below by \eqref{llnw} everywhere except a small measure set. The measure of this set tends to $0$ as $n\to\infty$. 
\end{proof}

\begin{definition}[Substitution map]
Let $\cB$ be a finite alphabet and consider some {\it substitution map} $\varrho\colon \cA\to \cB^\ast$ that maps any letter from $\cA$ into some finite word over the alphabet $\cB$. We extend this map to finite and bi-infinite sequences by concatenation in a natural way,
\begin{equation}\label{newdefsubs}
	\underline\varrho\colon\cA^\bZ\to\cB^\bZ,\quad
	\underline\varrho(\ldots,a_{-1}|a_0,a_1,\ldots)
	\eqdef(\ldots\varrho(a_{-1})|\varrho(a_0)\varrho(a_1)\ldots).
\end{equation}
\end{definition}

\begin{definition}[Bernoulli-coded measures]\label{defBerncod}
Given a substitution $\varrho\colon\cA\to\cB^\ast$ and $\bp\in\cM_{\rm B}(\cA^\bZ,\sigma_\cA)$, define by
\[
\kappa(\cA,\varrho,\bp) \eqdef \underline\varrho_\ast \bp
\]
a probability measure on $\cB^\bZ$, that we call {\it Bernoulli-coded  measure} (relative to $\varrho$ and $\bp$). We will also consider the probability measure
\[
	\kappa_{\rm inv}(\cA,\varrho, \bp)
	\eqdef \sum_{a\in \cA} \frac{1 }{|\varrho(a)|}
	\sum_{j=0}^{|\varrho(a)|-1}(\sigma_\cB^j\circ \underline\varrho)_\ast (\bp|_{[a]}),
\]
that has the additional property of being invariant under the shift  map $\sigma_\cB$. 
\end{definition}

\subsection{Suspension spaces}\label{secsusp}

Given a finite alphabet $\cA$ and a \emph{roof function} $R\colon \cA \to \bN$, consider the one-step 
extension $\underline R$ of $R$ to $\cA^\bZ$ defined by
\[
	\underline{R} \colon \cA^\bZ \to \bN, \quad
	\ua =(\ldots,a_{-1}|a_0, a_1, \dots) \in \cA^\bZ \mapsto \underline{R}(\ua) \eqdef R(a_0).
\]
The \emph{discrete-time suspension space} $\cS_{\cA,R}$ associated to $\cA$ and $R$ is the quotient space
\[
	\cS_{\cA,R}\eqdef (\cA^\bZ \times \bN_0)_\sim,
\]
where $\sim$ is the equivalence relation identifying $(\ua,s)$ with $(\sigma_\cA (\ua), s- \underline{R} (\ua))$. We agree to represent each class by its \emph{canonical representation} $(\underline a,s)$ with $s\in\{0,\ldots,\underline R(\underline a)-1\}$.
The \emph{suspension} of $\sigma_\cA$ by $\underline R$ is the map
\[
	\Phi_{\cA,R}\colon\cS_{\cA,R}\to\cS_{\cA,R},\quad
	\Phi_{\cA,R}(\underline a,s)
	\eqdef\begin{cases}
		(\underline a,s+1)&\text{ if }0\le s<\underline R(\underline a)-1,\\
		(\sigma_\cA(\underline a),0)&\text{ if }s=\underline R(\underline a)-1.
			\end{cases}
\]

Let us now define $\Phi_{\cA,R}$-invariant measures related to some Bernoulli measure $\bp$  in $\cA^\bZ$. 
Let $\fm$ be the counting measure on $\bZ$. Consider the \emph{suspension} of $\bp$ by $R$, defined by
\begin{equation}\label{susnu}
	\lambda_{\cA,R,\bp}
	\eqdef \frac{(\bp\times\fm)|_{\cS_{\cA,R}}}{(\bp\times\fm)(\cS_{\cA,R})}
	= \frac{(\bp\times\fm)|_{\cS_{\cA,R}}}{\sum_{a\in\cA}R(a)\bp([a])}.
\end{equation}
Note that this defines a $\Phi_{\cA,R}$-invariant and ergodic probability measure. 
Let us also consider another way of defining a $\Phi_{\cA,R}$-invariant measure. Given a set $A\subset\cA^\bZ$ and $s\in\bN_0$, let
\begin{equation}\label{susnutilde}
	\widetilde\lambda_{\cA,R,\bp}(A\times\{s\})
	\eqdef \sum_{a\in\cA}\frac{1}{R(a)}\bp(A\cap[a]).
\end{equation}
Note that this also defines a $\Phi_{\cA,R}$-invariant and ergodic probability measure.%
\footnote{Both definitions \eqref{susnu} and \eqref{susnutilde} have natural generalizations to arbitrary $\sigma_\cA$-invariant measures. However, here we focus entirely on Bernoulli measures.}  

Note that, unlike $\lambda$, the push forward of the measure $\widetilde\lambda_{\cA,R,\bp}$ under the projection to $\cA^\bZ$ by the map $(\underline a,s)\mapsto\underline a$ is again the measure $\bp$. This will make the measure $\widetilde \lambda$ more convenient to work with in Section \ref{secabscasc}, when we study cascades of substitution maps. However, in what follows we will need some properties of $\lambda$. We observe the following natural connection between these two measures.

\begin{lemma}
	$\widetilde\lambda_{\cA,R,\widetilde\bp}=\lambda_{\cA,R,\bp}$, where 
\begin{equation}\label{abovefor}
	\widetilde\bp=(\widetilde p_a)_{a\in\cA}
	\quad\text{ is given by }\quad
	\widetilde p_a
	=  \frac{R(a)p_a}{\sum_{b\in \cA} R(b)p_b}.
\end{equation}
\end{lemma}

For further reference, note the following obvious fact.

\begin{lemma}\label{susLBpre}
	If $R(\cdot)=\text{constant}$, then $\widetilde\bp=\bp$.
\end{lemma}

\begin{lemma}\label{susLB}
	For any Bernoulli measure $\bp$ on $\cA^\bZ$ and any roof function $R\colon\cA\to\bN$, the automorphism $(\cS_{\cA,R},\Phi_{\cA,R},\lambda_{\cA,R,\bp})$ is LB.  Analogously, $(\cS_{\cA,R},\Phi_{\cA,R},\widetilde\lambda_{\cA,R,\bp})$ is LB.
\end{lemma}

\begin{proof}
	By definition, the automorphisms $(\cS_{\cA,R},\Phi_{\cA,R},\lambda_{\cA,R,\bp})$ and $(\cA^\bZ,\sigma_\cA,\bp)$ are Kakutani equivalent. Hence, the assertion about $\lambda_{\cA,R,\bp}$ follows from Remark \ref{remarkLB}. As $\widetilde\lambda_{\cA,R,\bp}$ is also a suspension measure (though for a different base Bernoulli measure), the assertion holds for it, too.
\end{proof}

For further reference, let us recall Abramov's formula for suspensions,
\begin{equation}\label{eqAbramov}
	h(\Phi_{\cA,R},\lambda_{\cA,R,\bp})
	= \frac{h(\sigma_\cA,\bp)}{\sum_{a\in\cA} R(a)p_a}.
\end{equation}

\subsection{Suspension spaces associated to substitution maps}\label{seBernn}

Let $\varrho\colon\cA\to\cB^\ast$ now be a substitution map. In order to study associated Bernoulli-coded measures, it is convenient to study the suspension space with the particular roof function $R=|\varrho|$. Formula \eqref{abovefor} then rewrites as 
\begin{equation}\label{deftildep}
	\widetilde\bp=(\widetilde p_a)_{a\in\cA}
	\quad\text{ is given by }\quad
	\widetilde p_a
	=  \frac{|\varrho(a)|p_a}{\sum_{b\in \cA} |\varrho(b)|p_b}.
\end{equation}

Let us now turn to the shift space $\cB^\bZ$. Consider the following projection 
\begin{equation}\label{defPi}
	\Pi_\varrho\colon\cS_{\cA,R}\to\cB^\bZ,\quad
	\Pi_\varrho(\underline a,s)
	\eqdef \sigma_\cB^s(\underline\varrho(\underline a)).
\end{equation}
The following fact is then straightforward. 

\begin{lemma}\label{lema34b}
	$(\Pi_\varrho)_\ast\widetilde\lambda_{\cA,|\varrho|,\widetilde\bp}=(\Pi_\varrho)_\ast\lambda_{\cA,|\varrho|,\bp}=\kappa_{\rm inv}(\cA,\varrho,\widetilde\bp)$, where $\widetilde\bp$ is given by \eqref{deftildep}.
\end{lemma}

Lemma \ref{lema34b} justifies that, if we want to study $\lambda_{\cA, |\varrho|, \bp}$ with the help of Bernoulli-coded measures, then we need to look at the measure $\kappa_{\rm inv}(\cA, \rho, \widetilde\bp)$, not at $\kappa_{\rm inv}(\cA, \varrho, \bp)$. Nevertheless, the measure $\kappa_{\rm inv}(\cA, \varrho, \bp)$ will also be a very important tool, as we explain in Remark \ref{remmm}.

\begin{lemma}\label{lemffact}
$(\cB^\bZ,\sigma_\cB)$ is a topological factor of $(\cS_{\cA,R},\Phi_{\cA,R})$ by $\Pi_\varrho$.
\end{lemma}

The next result follows from Lemmas \ref{susLB} and \ref{lemffact} together with Remark \ref{remarkLB}.

\begin{lemma}\label{ergodicLB}
	The automorphism $(\cB^\bZ,\sigma_\cB,\kappa_{\rm inv}(\cA,\varrho,\bp))$ is LB.
\end{lemma}

\subsection{$\bar f$-distances between Bernoulli-coded measures}\label{secfbardistance}

In the following, we consider the $\bar f$-distance between the measures introduced above, either on $\cA^\bZ$ or on $\cB^\bZ$. To simplify notation, we simply write $\bar f$ in both cases.

The following Lipschitz property of substitution maps follows straightforwardly from the definition of $\bar f$-pseudometric.

\begin{lemma}\label{lemaLipschitz}
	Let $\cA$ and $\cB$ be two finite alphabets and $\varrho\colon \cA\to \cB^\ast$ be a substitution map. For every $\underline a,\underline a'\in\cA^\bZ$,
\[
	\bar f\big(\underline\varrho(\underline a),\underline\varrho(\underline a')\big)
	\le \frac{\max |\varrho|}{\min |\varrho|}
		\bar f(\underline a,\underline a').
\]	
\end{lemma}

The following observation that is also an immediate consequence of the definition.

\begin{lemma} \label{lemfact:bar}
$
\bar{f}\big( \kappa(\cA,\varrho,\bp), \kappa_{\rm inv}(\cA, \varrho, \bp)\big) =0.
$
\end{lemma}

Let us now study the dependence of $\kappa(\cA,\varrho,\bp)$ on the substitution $\varrho$ and on the probability vector $\bp$.

\begin{proposition}[Dependence on substitution] \label{prolem:Y}
Let $\cA$ and $\cB$ be two finite alphabets. Consider substitutions $\varrho\colon\cA\to\cB^\ast$ and $\varrho'\colon\cA\to\cB^\ast$  such that there is $C\in(0,1)$ so that for every $a\in\cA$ the word $\varrho'(a)$ is a substring of the word $\varrho(a)$ satisfying
\[
	(1-C)|\varrho(a)| < |\varrho'(a)|.
\]
Then for every Bernoulli measure $\bp$ it holds
\[
	\bar{f}\big( \kappa(\cA,\varrho,\bp), \kappa(\cA,\varrho', \bp)\big)< C.
\]
\end{proposition}

\begin{proof}
It suffices to note that from any $n$-word in the support of $\kappa(\cA,\varrho,\bp)$ we need to remove at most $nC$ symbols to obtain a word in the support of $\kappa(\cA,\varrho',\bp)$ and that by such removal the probability distribution is preserved.
\end{proof}

Let us now study the dependence of $\kappa(\cA,\varrho,\bp)$ on the probability vector $\bp$. The following result is a consequence of Lemmas \ref{lemshields} and \ref{lemaLipschitz}.

\begin{proposition}[Dependence on probability vector] \label{prop:P}
Let $\cA$ and $\cB$ be two finite alphabets, $\varrho\colon\cA\to\cB^\ast$ be a substitution map, and $\bp,\bp'$ be two probability vectors on $\cA$. Then
\[
	\bar{f}\big( \kappa(\cA,\varrho,\bp), \kappa(\cA,\varrho,\bp')\big)
	\leq \frac12 \frac {\max |\varrho|} {\min |\varrho|} \cdot D_\cA(\bp,\bp').
\]
\end{proposition}

We finish this section with the following observation.

\begin{lemma}\label{lemend}
	Let $\cA$ and $\cB$ be two finite alphabets, $\varrho\colon \cA\to \cB^\ast$ a substitution map, and
	 $\bp$  a Bernoulli measure on $\cA^\bZ$ given by a vector $(p_a)_{a\in\cA}$.  Then for every $r\in\bN$ and the alphabet $\cA'\eqdef \cA^r$,
\[
	\kappa(\cA,\varrho,\bp)
	= \kappa(\cA',\varrho',\bp'),
\]
where $\varrho'\colon\cA^r\to \cB^\ast$ and $\bp'$ are defined by
\[
	\varrho'(a_1,\ldots,a_r)
	\eqdef \varrho(a_1)\varrho(a_2)\cdots\varrho(a_r)
	\quad\text{ and }\quad
	\bp'(a_1,\ldots, a_r)
	\eqdef p_{a_1} \cdots  p_{a_r}.
\]
\end{lemma}

\section{Cascade of substitutions into a common symbolic space}\label{secabscasc}

In Section \ref{ssec4a}, we introduce a special cascade of substitution maps, so-called \emph{repeat-and-tail substitutions}, and an associated cascade of Bernoulli-coded measures. Under some appropriate control of tail lengths across the cascade, in Section \ref{ssec4b}, we provide key estimates between the obtained Bernoulli-coded measures. These estimates can be made ``uniform across'' the cascade using large deviation results, see Section \ref{ssec4c}. They provide us effective estimates for the $\bar f$-distance between the obtained Bernoulli-coded measures, see Section \ref{ssec4d}. One important fact here is that those estimates are uniform with respect to the Bernoulli measures in the initial alphabet.
Our final goal is Theorem \ref{theoprop:pathfinal} which provides the symbolic part of the main results in this paper, see Section \ref{ssecFurtherStru}. 

\subsection{Repeat-and-tail substitutions}\label{ssec4a}

Fix an initial finite alphabet $\cA$ and a sequence of natural numbers $(m_n)_{n=1}^\infty$, $m_n\ge2$. Define inductively the alphabets
\begin{equation}\label{defalphscA} 
	\cA_0\eqdef \cA,\quad
	\cA_n\eqdef (\cA_{n-1})^{m_n},
\end{equation}
Let $M_0\eqdef\card\cA_0$. 
Note that
\[
	M_n
	\eqdef \card\cA_n
	=(M_{n-1})^{m_n}.
\]
Each alphabet $\cA_n$ is a collection of finite words formed from the previous ones $\cA_{n-k}$. In particular, $\cA_n$ is a finite collection of finite $(m_1\cdot m_2\cdots m_n)$-words over the initial alphabet $\cA_0=\cA$. At the same time at each step we want to view $\cA_n$ as a new abstract alphabet, that is, each of its elements represents one ``letter to write new words''.
Let us consider the naturally associated ``respelling map'' 
\[
	\Sub_{n,n-1}
	\colon\cA_n\to(\cA_{n-1})^{m_n}.
\]

Analogously, for $k\le n$ we define inductively
\[
	\Sub_{n,k}
	\eqdef \Sub_{k+1,k}\circ\cdots\circ\Sub_{n,n-1}
	\colon\cA_n\to(\cA_k)^{m_{k+1}\cdots m_n}.
\]
In particular,
\[
	\Sub_{n,0}
	\colon\cA_n\to(\cA_0)^{m_1\cdots m_n}.
\]

Let us denote by $(\ldots | a_0^{(n)},a_1^{(n)}, \ldots)\in(\cA_n)^\bZ$ an element in this sequence space and extend those maps to bijections between the corresponding sequence spaces,
\begin{equation}\label{eq:tailaddn-kn}\begin{split}
	&\underline \Sub_{n,k}
	\colon (\cA_n)^\bZ\to\big(\cA_k^{m_{k+1}\cdots m_n}\big)^\bZ,
	\quad\\
	&\underline\Sub_{n,k}(\ldots | a_0^{(n)},a_1^{(n)}, \ldots)
	\eqdef (\ldots | \Sub_{n,k}(a^{(n)}_0) \Sub_{n,k}(a^{(n)}_1)\ldots)
\end{split}\end{equation}
and, in particular,
\begin{equation}\label{eq:tailaddn-1n}\begin{split}
	&\underline \Sub_{n,0}
	\colon (\cA_n)^\bZ\to\big(\cA_0^{m_1\cdots m_n}\big)^\bZ
	=\cA^\bZ,
	\quad\\
	&\underline\Sub_{n,0}(\ldots | a_0^{(n)},a_1^{(n)}, \ldots)
	\eqdef (\ldots | \Sub_{n,0}(a^{(n)}_0) \Sub_{n,0}(a^{(n)}_1)\ldots).
\end{split}\end{equation}

Denote by $\sigma_n\eqdef\sigma_{\cA_n}\colon(\cA_n)^\bZ\to(\cA_n)^\bZ$ the left shift over $\cA_n$.
One can check that the previous construction provides the following conjugations:

\begin{lemma}\label{newlemconjBer}
 $((\cA_n)^\bZ,\sigma_n)$ is topologically conjugate with $(\cA_k^\bZ,\sigma_k^{m_{k+1}\cdots m_n})$ by $\underline\Sub_{n,k}$.
\end{lemma}

\begin{definition}[Repeat-and-tail substitutions]
	Consider finite alphabets $\cA$ and $\cB$. Let $(m_n)_n$  be a sequence of natural numbers. Let $(\cA_n)_n$ be defined as in \eqref{defalphscA}. Consider a cascade of  \emph{tailing maps} $\bt_n\colon \cA_n\to\cB^\ast$, $n\in\bN$. Let $\varrho_0$ be any substitution map from $\cA_0=\cA$ to $\cB^\ast$. Define inductively the substitution map $\varrho_n\colon\cA_n\to\cB^\ast$ by
\[
	\varrho_n(a^{(n)})
	= \varrho_{n-1}(a^{(n-1)}_1)\ldots\varrho_{n-1}(a^{(n-1)}_{m_n})\bt_n(a^{(n)}),\quad
	\text{ for all }n\in\bN,
\]
for every $a^{(n)}\in\cA_n$, where $(a^{(n-1)}_1\ldots a^{(n-1)}_{m_n})\eqdef\Sub_{n,n-1}(a^{(n)})$. We then call $\{\varrho_n\colon\cA_n\to\cB^\ast\}_{ n\in\bN_0}$ a cascade of \emph{repeat-and-tail substitutions}.
\end{definition}

Like in \eqref{newdefsubs}, we extend the definition of the substitution to work also on the $\cA_n$-based symbolic space $(\cA_n)^\bZ$.

Fix some Bernoulli measure $\bp$ on $\cA^\bZ$ and let
\begin{equation}\label{conconj}
	\bp_n
	\eqdef \big(\underline\Sub_{n,0}^{-1}\big)_\ast\bp.
\end{equation}
One can check that $\bp_n$ is again a Bernoulli measure on $(\cA_n)^\bZ$ (with probabilities given by Lemma \ref{lemend}). 

Together with the cascade of substitution maps $\varrho_n$, we consider the associated cascade suspension of suspension spaces $\cS_n=\cS_{\cA_n,|\varrho_n|}$ and the measures 
$\lambda_n(\bp)=\lambda_{\cA_n,|\varrho_n|,\bp_n}$. Again following Section \ref{seBernn}, we also consider the associated Bernoulli measure $\widetilde{\bp_n}$ defined as in \eqref{deftildep}, which now takes the form
\begin{equation}\label{conconjtilde}
 	\widetilde{\bp_n}([a^{(n)}])
	= \frac {|\varrho_n(a^{(n)})| \, \bp_n([a^{(n)}])} {\sum_{b\in \cA_n} |\varrho_n(b)| \,\bp_n([b])}
	\quad\text{ for any }a^{(n)}\in\cA_n.
\end{equation}

\begin{remark}\label{remmm}	
 Observe that, in general, $\widetilde{\,\bp_n\,}$ does not have the ``self-similar structure'' of $\bp_n$ given in \eqref{conconj}, that is,
\[
	\widetilde{\bp_n}
	\neq \big(\underline\Sub_{n,0}^{-1}\big)_\ast\widetilde\bp,
\]
with $\widetilde\bp$ defined in \eqref{deftildep}.
For this reason, it is not very convenient to study $\widetilde{\bp_n}$ directly. However, as we will show in Corollary \ref{corthisisit}, we can study with $\widetilde{\bp_n}$ with the help of $\bp_n$. It will turn out that, under Assumption \ref{assumption1} that we introduce below, they are asymptotically comparable for large $n$, as will be stated in Corollary \ref{corthisisit}.
\end{remark}

We state the following auxiliary result for further reference.

\begin{lemma} \label{lem:path1}
	For every $\varepsilon>0$ and $n\in\bN$ there is $\delta>0$ such that $D_{\cA_n}(\bp_n,\bp_n') < \varepsilon$ for any two Bernoulli measures $\bp, \bp'\in \cM_{\rm B}(\cA^\bZ,\sigma_\cA)$ satisfying $D_\cA(\bp,\bp')< \delta$.
\end{lemma}

\begin{proof}
By Lemma \ref{newlemconjBer}, $((\cA_n)^\bZ,\sigma_n)$ and $(\cA^\bZ,\sigma_\cA^{m_1\cdots m_n})$ are topologically conjugate by $\underline\Sub_{n,0}$. Hence, any Bernoulli measure $\bp$ maps to the Bernoulli measure $\bp_n$ continuously, see \eqref{conconj}. Moreover, this map is uniformly continuous on the simplex of all Bernoulli measures. 
\end{proof}

\subsection{Control of tail lengths}\label{ssec4b}

We will make throughout this section the following assumption on the repeat-and-tail substitutions $\{\varrho_n\colon\cA_n\to\cB^\ast\}_{ n\in\bN_0}$, we use:

\begin{assumption}[Control of tail lengths]\label{assumption1}
It holds
\begin{equation}\label{maxismin}
\max|\varrho_0|=\min|\varrho_0|.
\end{equation}
There exists $K>0$ such that the length of the tailing map is uniformly bounded for all $n$ and $a^{(n)}\in \cA_n$ by
\[
	|\bt_n(a^{(n)})|
	\leq K 2^{-n} \sum_{i=1}^{m_n} |\varrho_{n-1}(a_i^{(n-1)})|.
\]
\end{assumption}

\begin{remark}
	The assumption \eqref{maxismin} is put mainly to simplify some part of our exposition.
\end{remark}

We now derive some preliminary results about the control of the ``tail lengths''.
	 
\begin{lemma}\label{lemma:tailslength}
Under Assumption \ref{assumption1}, for every $k\in\bN_0$, $n> k$, and $a^{(n)}\in \cA_n$, letting $(a^{(k)}_1,\ldots ,a^{(k)}_{m_{k+1}\cdots m_n})\eqdef \Sub_{n,k}(a^{(n)})$, we have
\[
		1
	\le \frac{\big|\varrho_n(a^{(n)})\big|}
		{\big|\varrho_k(a^{(k)}_1)\ldots \varrho_k(a^{(k)}_{m_{k+1}\cdots m_n})\big|}
	\le 1+4K(2^{-k}-2^{-n}),
\]
where $K$ is as in Assumption \ref{assumption1}.
\end{lemma}	 
	  
\begin{proof}
Given $k\in\bN_0$, let $n=k+1$. Given any $a^{(k+1)}\in\cA_{k+1}=(\cA_k)^{m_{k+1}}$, by means of the substitution
$\Sub_{k+1,k}$
 in \eqref{eq:tailaddn-kn}, we can write 
\[
	\Sub_{k+1,k}(a^{(k+1)})=(a^{(k)}_1,\ldots ,a^{(k)}_{m_{k+1}}).
\]	
By Assumption \ref{assumption1}, $\varrho_k(a^{(k)}_1)\ldots\varrho_{k}(a^{(k)}_{m_{k+1}})$ is a prefix of  $\varrho_{k+1}(a^{(k+1)})$ and
\[\begin{split}
	\sum_{i=1}^{m_{k+1}} |\varrho_k(a_i^{(k)})|
	&\le |\varrho_{k+1}(a^{(k+1)})|
	= \sum_{i=1}^{m_{k+1}} |\varrho_k(a_i^{(k)})| + |t_{k+1}(a^{(k+1)})|\\
	&\le (1+K2^{-(k+1)})\sum_{i=1}^{m_{k+1}} |\varrho_k(a_i^{(k)})| .
\end{split}\]
This implies that
\[
	1
	\le \frac{\big|\varrho_{k+1}(a^{(k+1)})\big|}
		{\big|\varrho_k(a^{(k)}_1)\ldots\varrho_{k}(a^{(k)}_{m_{k+1}})\big|}
	\le 1+K2^{-(k+1)},
\]
proving the assertion for $n=k+1$. Using the above steps repeatedly, for every $n> k$ and $(a^{(k)}_1,\ldots ,a^{(k)}_{m_{k+1}\cdots m_n})\eqdef \Sub_{n,k}(a^{(n)})$ we get that
\[
	1
	\le \frac{\big|\varrho_n(a^{(n)})\big|}
		{\big|\varrho_k(a^{(k)}_1)\ldots \varrho_k(a^{(k)}_{m_{k+1}\cdots m_n})\big|}
	\le \prod_{i=k+1}^n(1+K2^{-i})
	\le 1+4K(2^{-k}-2^{-n}),	
\]
as claimed.
\end{proof}

\begin{corollary} \label{cor:maxmin}
Under Assumption \ref{assumption1}, for every $n\in\bN$ we have
\[
0\le \frac {\max|\varrho_n| - m_1 \cdots m_n \cdot \min |\varrho_0|} {\max|\varrho_n|} \leq \frac {4K} {1+4K}.
\]
where $K$ is as in Assumption \ref{assumption1}. In particular,
\[
1 \leq \frac {\max |\varrho_n|} {\min |\varrho_n|} \leq 1+4K,
\]
\end{corollary}

\begin{proof}
Taking $k=0$ in Lemma \ref{lemma:tailslength}, we get 
\[
	m_1 \cdots m_n  \min|\varrho_0| 
	\leq \min |\varrho_n| 
	\leq \max |\varrho_n| 
	\leq m_1 \cdots m_n  \max|\varrho_0|  (1+4K(1-2^{-n})).
\]
This immediately implies the assertions.
\end{proof}

Recall the definition of the measure $\widetilde{\bp_n}$ in \eqref{conconjtilde}. Let us first derive some estimate about their maximal ``relative growth'' in $n$. 

\begin{corollary}\label{corclavonhieraus}
	 Under Assumption \ref{assumption1}, 
	 for every $n\in\bN$, $m\le m_1\cdots m_n$, and $a_1,\ldots,a_m\in\cA$, 
\[
	1-4K
	\le \frac{(\underline\Sub_{n,0})_\ast\widetilde{\,\bp_n\,} ([a_1,\ldots, a_m])}
		{\widetilde{\,\bp\,}([a_1,\ldots, a_m])}
	\le 1+4K,
\]	
where $K$ is as in Assumption \ref{assumption1}.
Moreover, 
\[
	D_{\cA_n}(\widetilde{\,\bp_n\,},\bp_n)<4K.
\]	
\end{corollary}

\begin{proof}
First note that $\max|\varrho_0|=\min|\varrho_0|$ in Assumption \ref{assumption1} implies that $\widetilde\bp=\bp$, see Lemma \ref{susLBpre}. Then \eqref{conconj} implies that $(\underline\Sub_{n,0}^{-1})_\ast\widetilde\bp=(\underline\Sub_{n,0}^{-1})_\ast\bp=\bp_n$.
From the definition of $\widetilde{\,\bp_n\,}$ in \eqref{conconjtilde} and Corollary \ref{cor:maxmin}, given $a^{(n)}\in\cA_n$, we get
\[
	\frac{\widetilde{\,\bp_n\,}([a^{(n)}])}{(\underline\Sub_{n,0}^{-1})_\ast\widetilde\bp([a^{(n)}])}
	=
	\frac{\widetilde{\,\bp_n\,}([a^{(n)}])}{1}\frac{1}{\bp_n([a^{(n)}])}
	= \frac{1 }{\sum_{b\in \cA_n}\frac{ |\varrho_n(b)| }{|\varrho_n(a)|}\,\bp_n([b])}
	\le 1+4K.
\]
The lower bound is analogous.
As each cylinder $[a_1,\ldots,a_m]\subset\cA^\bZ$, $m\le m_1\cdots m_n$, is a finite disjoint union of cylinders $[a^{(n)}]$, $a^{(n)}\in\cA_n$, the first assertion follows.

For the second assertion, check that, with the above
\[\begin{split}
	D_{\cA_n}(\widetilde{\,\bp_n\,},\bp_n)
	&= \sum_{a^{(n)}\in\cA_n}|\widetilde{\,\bp_n\,}([a^{(n)}])-\bp_n([a^{(n)}])|\\
	&= \sum_{a^{(n)}\in\cA_n}\bp_n([a^{(n)}])
		\Big|\frac{\widetilde{\,\bp_n\,}([a^{(n)}])}{\bp_n([a^{(n)}])}-1\Big| \\
	&\le 4K \sum_{a^{(n)}\in\cA_n}\bp_n([a^{(n)}])
	=4K.
\end{split}\]
This finishes the proof.
\end{proof}

Recall the definition of the measure $\bp_n$ in \eqref{conconj}. Our goal in the remainder of this subsection is to estimate the $\bar f$-distance between $\bp_n$ and $\widetilde{\,\bp_n\,}$, for $n$ large enough. Note that for every $a\in\cA_n$,
\[\begin{split}
	\big| \widetilde{\bp_n}([a])-\bp_n([a])\big|
	&= \Big| \frac{|\varrho_n(a)|\bp_n([a])}{\sum_{b\in \cA_n} |\varrho_n(b)|\bp_n([b])}-\bp_n([a])\Big|\\
	&= \bp_n([a])\frac{1}{\sum_{b\in \cA_n} |\varrho_n(b)| \bp_n([b])}\,
		\Big| |\varrho_n(a)|-{\sum_{b\in \cA_n} |\varrho_n(b)|\bp_n([b])}\Big|.
\end{split}
\]
To that end, consider the ``expected roof length''
\[
	\bE_n(\bp)
	\eqdef \sum_{b\in \cA_n} |\varrho_n(b)|\,\bp_n([b]).
\]
Given $\underline a\in\cA^\bZ$, $\underline a^{(k)}=\underline\Sub_{k,0}^{-1}(\underline a)\in(\cA_k)^\bZ$, and $n> k$, consider the ``normalized fluctuations''
\begin{equation}\label{defDeltank}
	\Delta_{n,k}(\bp,\underline a)
	\eqdef \frac{\left|\bE_k(\bp)-
		\frac{1}{{m_{k+1}\cdots m_n} }
			\big|\varrho_{k}(a_0^{(k)})\ldots\varrho_{k}(a_{m_{k+1}\cdots m_n-1}^{(k)})\big|
	\right|}
		{\bE_k(\bp)}
\end{equation}			
For $k=n$, taking $\underline a^{(n)}=\underline\Sub_{n,0}^{-1}(\underline a)$, we also let
\begin{equation}\label{defDeltann}
	\Delta_{n,n}(\bp,\underline a)
	\eqdef \frac {\left|\bE_n(\bp)-|\varrho_n(a_0^{(n)})|\right|} 
			{ \bE_n(\bp)}.
\end{equation}
Note that this object is a piecewise constant  function of $\underline a$ (constant on cylinders of level	$m_1\cdots m_n$).
With the above, we get
\begin{equation} \label{eqn:delnn}
	\big| \widetilde{\bp_n}([a])-\bp_n([a])\big|
	= \bp_n([a])\cdot\Delta_{n,n}(\bp, \underline a).
\end{equation}
	
We obtain the following estimates from Lemma~\ref{lemma:tailslength}. 

\begin{corollary}\label{corend}
Under Assumption \ref{assumption1}, for every point $\underline a\in\cA^\bZ$ and every $n>k$ it holds
\[
	\Delta_{n,n}(\bp,\underline a)
	\le \Delta_{n,k}(\bp,\underline a)\big(1+4K2^{-k}\big) +4K2^{-k}.
\]
\end{corollary}

\begin{proof}
Let us first show that for every $k\in\bN$ and $n>k$,
\begin{eqnarray}
	1
	&\le& \displaystyle\frac{\sum_{a\in\cA_n}|\varrho_n(a)|\bp_n([a])}
		{(m_{k+1}\cdots m_n)\sum_{b\in \cA_{k}} |\varrho_{k}(b)|\,\bp_{k}([b])}
	= \frac{\bE_n(\bp)}{(m_{k+1}\cdots m_n)\bE_k(\bp)}\label{eqnarray1}\\
	&\le& 1+4K(2^{-k}-2^{-n}).\notag	
\end{eqnarray}	
Observe that the numerator in \eqref{eqnarray1} is the expected value of the Birkhoff averages of the function $(\cA_n)^\bZ\ni\underline a^{(n)}\mapsto |\varrho_n(a_0^{(n)})|$ with respect to $(\sigma_n,(\cA_n)^\bZ,\bp_n)$. On the other hand, the denominator in \eqref{eqnarray1} is the expected value of the Birkhoff averages of the function $(\cA_k)^\bZ\ni\underline b^{(k)}\mapsto|\varrho_k(b_0^{(k)})|$ with respect to $(\sigma_k^{m_{k+1}\cdots m_n},(\cA_k)^\bZ,\bp_k)$ (note that this is also an ergodic automorphism as  we are dealing with Bernoulli measures). With this observation, the assertion is then an immediate consequence of Lemma \ref{lemma:tailslength}.

Using now \eqref{eqnarray1} together with again Lemma \ref{lemma:tailslength}, we get
\[\begin{split}
	& \frac {|\varrho_n(a^{(n)})| - \bE_n(\bp)} 
			{\bE_n(\bp)}\\
	&\quad\quad\le	\frac{\frac{1}{m_{k+1}\cdots m_n}\big|\varrho_{k}(a_0^{(k)})\ldots\varrho_{k}(a_{m_{k+1}\cdots m_n-1}^{(k)})\big|\big(1+4K2^{-k}\big)
		- \bE_k(\bp)}
	{\bE_k(\bp)},
\end{split}\]
together with analogous lower bounds. This implies the assertion.
\end{proof}

\subsection{Some large deviation results}\label{ssec4c}

The following result provides a large deviation result that is \emph{uniform across} all Bernoulli measures. 
We present it in a broader context, initially unrelated to the preceding content.
In Section \ref{ssecFurtherStru}, it will be implemented to describe further structures.

\begin{proposition}[Uniform Law of Large Numbers] \label{proplem:bernsteinagain}
Let $\cC$ be a finite alphabet. For every $\delta>0$ there exist $L=L(\delta)$ such that for every Bernoulli measure $\bp$ on $\cC^\bZ$ the set
\begin{multline*}
G(\bp,L, \delta) 
\eqdef\\ \Big\{\underline b\in \cC^\bN\colon 
	 \big|\bp([c])-\frac 1\ell \card\{i=0,\ldots,\ell-1\colon b_i=c\}\big|< \frac L\ell+\delta
	 \text{ for all }\ell\in\bN,c\in\cC\Big\}.
\end{multline*}
satisfies $\bp(G(\bp,L,\delta))> 1-\delta$.
\end{proposition}

\begin{proof}
Fix any
$\bp$. Given $\ell\in\bN$, $c\in\cC$, and $L>0$ define
\[
 G_{\ell,c}(\bp,L,\delta) \eqdef
\Big\{\underline b\in \cC^\bN\colon 
	 \frac 1\ell\big|\ell\bp([c])- \card\{i=0,\ldots,\ell-1\colon b_i=c\}\big|\ge \frac L\ell+\delta
\Big\}.\]
Note that
\[
	G(\bp,L,\delta) 
	= \cC^\bZ \setminus \bigcup_{\ell\in\bN} \bigcup_{c\in \cC} G_{\ell,c}(\bp,L,\delta),
\]
and hence
\begin{equation} \label{eqn:bernag0}
	\bp( G(\bp,L,\delta)) 
	\geq 1- \sum_{\ell\in\bN} \sum_{c\in \cC} \bp( G_{\ell,c}(\bp,L,\delta)).
\end{equation}

Let us recall a simple form of the Bernstein inequality that we are going to apply: for a sequence $(X_i)_i$ of independent and identically distributed copies of a random variable $X$ with mean value equal to 0, for every $n\in\bN$ we have
\begin{equation}\label{eqBern}
	\bP\Big(|\frac 1\ell \sum_{i=1}^{\ell} X_i|>a\Big) 
	\leq 2 \exp\Big(-\frac {\ell a^2/2} {\bE(X^2) + ||X||a/3} \Big) .
\end{equation}

Given $c\in\cC$, consider now the random variable $X$ over the sigma field of $\cC^\bZ$ that takes the value $1-\bp([c])$ with probability $\bp([c])$ and the value $-\bp([c])$ with probability $1-\bp([c])$. Let $(X_i)_i$ be a sequence of independent and identically distributed copies of $X$. Note that, in our context, we have 
\[\begin{split}
	\frac{1}{\ell} \sum_{i=1}^\ell X_i
	&= \frac 1\ell\big(\card\{i=0,\ldots,\ell-1\colon b_i=c\}-\ell\bp([c])\big),\\
	\bE(X)
	&=\big(1-\bp([c])\big)\bp([c])-\bp([c])\big(1-\bp([c])\big)=0,\\
	\bE(X^2)
	&= \big(1-\bp([c])\big)^2\bp([c])+\big(-\bp([c])\big)^2\big(1-\bp([c])\big)
	= \bp([c])(1-\bp([c])),\\
	\lVert X\rVert
	&=\max\big\{1-\bp([c]),\bp([c])\big\}.
\end{split}\]
Applying \eqref{eqBern}, we get
\[\begin{split}
\bp( G_{\ell,c}&(\bp,L,\delta))\\
&\leq 2\exp\Big(-\frac {\ell(\delta+L/\ell)^2/2} {\bp([c])(1-\bp([c])) + \max\{\bp([c]), 1-\bp([c])\} \cdot (\delta+L/\ell)/3}  \Big)\\
&\leq 2 \exp\Big(- \frac {6(\ell\delta+L)^2} {3\ell +4(\ell\delta+ L)}\Big).
\end{split}\]
Thus,
\begin{equation} \label{eqn:bernag}
\sum_{\ell\in\bN} \sum_{c\in \cC} \bp( G_{\ell,c}(\bp,L,\delta))
\leq 2\card\cC \cdot 
	\sum_{\ell=1}^\infty \exp\Big(-\frac {6(\ell\delta+L)^2} {3\ell +4(\ell\delta+ L)}\Big).
\end{equation}
Check that $a/(b+c)\ge\min\{a/2b,a/2c\}$ implies
\[
	\min\Big\{\ell\delta^2,\frac34\ell\delta\Big\}
	\le \frac {6(\ell\delta+L)^2} {3\ell +4(\ell\delta+ L)}
	\to\infty
\]
as $L\to\infty$. Thus, the series on the right-hand side in \eqref{eqn:bernag} is summable and uniformly bounded from above by 
\[
	\sum_{\ell=1}^\infty\exp\Big(-\min\big\{\ell\delta^2,\frac34\ell\delta\}\Big)
	<\infty.
\]
Moreover, its summands pointwise converge to $0$ as $L\to\infty$. Hence, the sum in \eqref{eqn:bernag} converges to 0 as $L$ increases monotonically. It suffices now to choose $L_0(\delta)>0$ sufficiently large. Substituting this to \eqref{eqn:bernag0}, we get the assertion. Note that $L_0(\delta)$ does not depend on the particular choice of $\bp$.
\end{proof}

Let us draw some consequences from the above proposition in our current context. Recall the definition of the (normalized) fluctuations $\Delta_{n,n}$ in \eqref{defDeltann}.

\begin{proposition}\label{prousedcorollary}
Under Assumption \ref{assumption1}, for every $\varepsilon>0$ there exists $n_0\in\bN$ such that for every Bernoulli measure $\bp$ on $\cA^\bZ$ and every $n> n_0$, we get
\[
	\bp\big( \big\{\underline a\in \cA^\bZ\colon \Delta_{n,n}(\bp,\underline a)
	 \leq \varepsilon  \big\}\big)
	 >1-\varepsilon.
\]
\end{proposition}

\begin{proof}
Choose $k\in\bN$ sufficiently large and $\delta>0$ sufficiently small, to be specified later. Apply Proposition \ref{proplem:bernsteinagain} to the alphabet $\cA_{k}$ and let $L_k=L(k,\delta)$ be as provided by this proposition. Fix some Bernoulli measure $\bp$ on $\cA^\bZ$ and let $\bp_{k}$ be as in \eqref{conconj}. 
Let $G_k=G(\bp_{k}, L_k, \delta)\subset(\cA_k)^\bZ$ as defined in Proposition \ref{proplem:bernsteinagain}.  
By this proposition, $\bp_k(G_k)>1-\delta$ and for every $\underline a^{(k)}\in G_k$, $\ell\in\bN$, and $c\in\cA_k$,
\begin{equation}\label{eqestPoin}
	 \big|\bp_k([c])-\frac 1\ell \card\{i=0,\ldots,\ell-1\colon a_i^{(k)}=c\}\big|
 	< \frac{L_k}{\ell}+\delta.
\end{equation}
Note that $|\varrho_{k}(a_0^{(k)})\ldots\varrho_{k}(a_{m_{k+1}\cdots m_n-1}^{(k)})|$ is just the Birkhoff sum (relative to $\sigma_k$) of the piecewise constant, and hence continuous, function 
\[	
	g\colon(\cA_k)^\bZ\to\bN,\quad
	g(\underline a^{(k)})
	\eqdef|\varrho_{k}( a_1^{(k)})|.
\]
Taking now $\ell=m_{k+1}\cdots m_n$, let us invoke the estimate \eqref{eqestPoin}. For  all $\underline a^{(k)} \in G_k$, we get
\[\begin{split}
	\big|&\varrho_{k}(a_0^{(k)})\ldots\varrho_{k}(a_{m_{k+1}\cdots m_n-1}^{(k)})\big|	\\
	&=\sum_{i=0}^{m_{k+1}\cdots m_n-1}|\varrho_{k}(a_i^{(k)})|	
	= \sum_{c\in\cA_k}|\varrho_k(c)|\card\{i\colon a_i^{(k)}=c\}	\\
	&= m_{k+1}\cdots m_n\sum_{c\in\cA_k}|\varrho_k(c)|
		\frac{1}{m_{k+1}\cdots m_n}\card\{i\colon a_i^{(k)}=c\}	\\
	&= 	m_{k+1}\cdots m_n\sum_{c\in\cA_k}
	|\varrho_k(c)|\big(\bp_k([c])+\xi(c)\big).
\end{split}\]
In the latter expression, by \eqref{eqestPoin}, the term $\xi(c)$ denotes the deviation of the frequency of the symbol $c$ from its ``expected value'' $\bp_k([c])$. Note that it satisfies 
\begin{equation}\label{laaater}
	|\xi(c)|\le \frac{L_k}{m_{k+1}\cdots m_n}+\delta,\quad
	\sum_{c\in\cA_k}\xi(c)=0.
\end{equation}
Hence, for any $\underline a\in \cA^\bZ$ such that $\underline a^{(k)}=\underline\Sub_{k,0}^{-1}(\underline a)\in G_k$,
 we can estimate the numerator of \eqref{defDeltank} as follows
\[\begin{split}
	\Delta_{n,k}(\bp,\underline a)\bE_k(\bp)
	&=\Big| \sum_{c\in\cA_k}|\varrho_k(c)| \big(\bp_k([c])+\xi(c)\big)-\bE_k(\bp)\Big|\\
	&=\Big| \sum_{c\in\cA_k}|\varrho_k(c)| \big(\bp_k([c])+\xi(c)\big)
		-\sum_{c\in\cA_k}|\varrho_k(c)|\bp_k([c])  \Big|\\
	&=\Big|\sum_{c\in\cA_k}|\varrho_k(c)|\xi(c) \Big|\\
	&=\Big|\sum_{c\in\cA_k\colon\xi(c)>0} |\varrho_k(c)|\xi(c)
		+ \sum_{c\in\cA_k\colon\xi(c)\le 0} |\varrho_k(c)|\xi(c) \Big|\\
	&\le \card\cA_k\max|\varrho_k|\max|\xi|	- \card\cA_k\min|\varrho_k|\max|\xi|	\\
	\text{\tiny{(using \eqref{laaater})}}\quad
	&\le \card\cA_k(\max|\varrho_k|-\min|\varrho_k|)\big(\frac{L_k}{m_{k+1}\cdots m_n}+\delta\big).
\end{split}\]
Using that $\bE_k(\bp)\ge \min|\varrho_k|$, for $n>k$ we get
\[
	\Delta_{n,k}(\bp,\underline a)
	\le \card\cA_k
		\big(\frac{\max|\varrho_k|}{\min|\varrho_k|}-1\big)
		\big(\frac{L_k}{m_{k+1}\cdots m_n}+\delta\big).
\]
Hence, together with Corollaries \ref{corend} and \ref{cor:maxmin}, we get
\begin{eqnarray}
	\Delta_{n,n}(\bp,\underline a)
	&\le& \card\cA_k 
		\Big(
			(1+4K)-1\Big)
		\Big(\frac{L_k}{m_{k+1}\cdots m_n}+\delta\Big)
		(1+4K2^{-k}) \notag\\
	&\phantom{\le}&+4K2^{-k}\notag\\
	&=& 4K2^{-k}+\label{fomrla1}\\
	&\phantom{=}&+\delta\cdot\card\cA_k\Big(
		(1+4K)-1\Big)(1+4K2^{-k})\label{fomrla2}\\
	&\phantom{=}&+\frac{L_k}{m_{k+1}\cdots m_n}\cdot\card\cA_k
		\Big(
			(1+4K)-1\Big)(1+4K2^{-k})\label{fomrla3}
\end{eqnarray}

Now let us argue about the order in which the above constants are to be chosen. First, we choose $k$ large such that \eqref{fomrla1} is small. This determines the term $\card\cA_k$. Next, we choose $\delta$ such that \eqref{fomrla2} is small. This determines $L_k=L(k,\delta)$. Finally, let $n_0$ such that \eqref{fomrla3}  is small for all $n\ge n_0$. This way, we guarantee that $\Delta_{n,n}(\bp,\underline a)$ is smaller than $\varepsilon$. 

Recall that, the above holds for any $\underline a^{(k)}\in G_k$ and that $\bp_k(G_k)>1-\delta$. Hence, it follows that $\Delta_{n,n}(\bp,\underline a)<\varepsilon$ holds for any $\underline a\in \underline\Sub_{k,0}(G_k)$ and that $\bp(\underline\Sub_{k,0}(G_k))>1-\delta$.

Note that all choices are independent of $\bp$, in particular $n$ is independent of $\bp$. 
\end{proof}

\begin{corollary}\label{corthisisit}
Under Assumption \ref{assumption1}, for every $\varepsilon>0$ there exists $n_0\in\bN$ such that for every Bernoulli measure $\bp$ on $\cA^\bZ$ and every $n> n_0$  the measure $\bp_n=(\underline\Sub_{n,0}^{-1})_\ast\bp$ and the corresponding measure $\widetilde{\bp_n}$ defined in \eqref{conconjtilde} satisfy 
\[
	D_{\cA_n}(\widetilde{\bp_n},\bp_n)
	= \sum_{a\in\cA_n}\big|\widetilde{\bp_n}([a])-\bp_n([a])\big|
	<\varepsilon.
\]
\end{corollary}

\begin{proof}
We recall the formula \eqref{eqn:delnn}, which will now be finally used.
By Proposition \ref{prousedcorollary}, given $\varepsilon>0$ there exists $n_0\in\bN$, such that for any Bernoulli measure $\bp$ there is a set $G_\bp\subset\cA^\bZ$ such that $\bp(G_\bp)>1-\varepsilon$ and $ \Delta_{n,n}(\bp,\underline a)\leq \varepsilon$ for any $n>n_0$ and every $\underline a\in G_\bp$. Hence, with $G_{n,\bp}=\underline\Sub_{n,0}^{-1}(G_\bp)$, we have $\bp_n(G_{n,\bp} )>1-\varepsilon$. For every $a\in\cA_n$ such that $[a]\cap G_{n,\bp}\ne\emptyset$ we get
\[
	\big| \widetilde{\bp_n}([a])-\bp_n([a])\big|\le\varepsilon.
\]
Hence, we get 
\begin{equation*}\begin{split}
	\sum_{a\in\cA_n}\big| \widetilde{\bp_n}([a])-&\bp_n([a])\big|\\
	&= \sum_{[a]\cap G_{n,\bp}\ne\emptyset}\big| \widetilde{\bp_n}([a])-\bp_n([a])\big|
		+ \sum_{[a]\cap G_{n,\bp}=\emptyset}\big|\widetilde{\bp_n}([a])-\bp_n([a])\big|\\
	&\le\varepsilon\cdot\bp_n(G_{n,\bp})+2(1-\bp_n(G_{n,\bp}))\\
	&<\varepsilon+2\varepsilon.
\end{split}\end{equation*}
This implies the assertion.
\end{proof}

\subsection{$\bar f$-convergence}\label{ssec4d}

We study the $\bar f$-distance between the elements of the sequence 
\begin{equation}\label{defnunp}
	\nu_n(\bp)
	\eqdef \kappa_{\rm inv}(\cA_n,\varrho_n,\widetilde{\bp_n}).
\end{equation}
We start with two preliminary results that will be implemented below.

\begin{proposition}\label{prokickoff}
	Under Assumption \ref{assumption1}, 
	for any $n\in\bN$ and $\bp\in\cM_{\rm B}(\cA^\bZ,\sigma_\cA)$,
\[
	\bar f\big(\nu_n(\bp),\nu_0(\bp)\big)
	\le 6K+8K^2,
\]	
where $K$ is as in Assumption \ref{assumption1}.
\end{proposition}

\begin{proof}
Recall that, by Lemma \ref{susLBpre}, our hypothesis $\max|\varrho_0|=\min|\varrho_0|$ in Assumption \ref{assumption1} implies $\widetilde\bp=\bp$. Besides \eqref{defnunp}, recall that 
\[
	\nu(\bp) 
	= \kappa_{\rm inv}(\cA, \varrho_0, \widetilde \bp) 
	= \kappa_{\rm inv}(\cA, \varrho_0, \bp)
	= \kappa_{\rm inv}(\cA_n, \varrho', \bp_n)
\]	
(where $ \varrho'$ is like in Lemma \ref{lemend}).
Thus, 
\[\begin{split}
	\bar f( \nu_n(\bp),  \nu(\bp))
	&\leq \bar f \big( \kappa_{\rm inv}(\cA_n, \varrho_n, \widetilde{\bp_n}), \kappa_{\rm inv}(\cA_n, \varrho_n, \bp_n)\big) \\
	&\phantom{\le}+ \bar f\big ( \kappa_{\rm inv}(\cA_n, \varrho_n,\bp_n),  \kappa_{\rm inv}(\cA_n, \varrho',\bp_n)\big)\\
	\text{\tiny{(using Propositions \ref{prop:P} and \ref{prolem:Y})}}\quad
	&\le  \frac 1 2 \frac{\max|\varrho_n|}{\min|\varrho_n|} D_{\cA_n}(\bp_n, \widetilde{\bp_n}) 
		+ \frac {4K} {1+4K}\\		
	\text{\tiny{(by Corollary \ref{cor:maxmin}
	)}}\quad
	&< 	\frac 1 2 (1+4K)  D_{\cA_n}(\bp_n, \widetilde{\bp_n}) 
			+ 4K.
\end{split}\]
Note that Corollary \ref{corclavonhieraus} implies 
$
	D_{\cA_n}(\bp_n, \widetilde{\bp_n}) 
	\le 4K.
$
This proves the assertion.
\end{proof}

\begin{lemma}\label{lemprop:earlier}
Under Assumption \ref{assumption1}, for any $\varepsilon>0$, there is $n_0\in\bN$ such that for all $n\ge n_0$ and all $\bp\in\cM_{\rm B}(\cA^\bZ,\sigma_\cA)$, it holds 
\[\bar{f}\big( \kappa(\cA_n, \varrho_n, \widetilde{\bp_n}), \kappa(\cA_n, \varrho_n, \bp_n)\big)<\varepsilon.\]
\end{lemma}

\begin{proof}
By Corollary \ref{corthisisit}, given $\varepsilon>0$ there is $n_0\in\bN$ such that for all $n> n_0$ and all $\bp$ we get
\[\begin{split}
	D_{\cA_n}(\widetilde{\bp_n},\bp_n)
	<\varepsilon.
\end{split}\]
Applying Proposition \ref{prop:P} to $\bp=\bp_n$ and $\bp'=\widetilde{\bp_n}$, we get
\[\begin{split}
	\bar{f}(\kappa(\cA_n, \varrho_n, \bp_n), \kappa(\cA_n, \varrho_n, \widetilde{\bp_n}))
	&\leq \frac12 \frac {\max |\varrho_n|} {\min |\varrho_n|} \cdot D_{\cA_n}(\bp_n,\widetilde{\bp_n})\\
	{\tiny{\text{(together with Corollary \ref{cor:maxmin})}}}\quad
	&< \frac12
	(1+4K)\cdot \varepsilon.
\end{split}\]
This implies the assertion.
\end{proof}

The main result in this section is the following. 

\begin{theorem}\label{theCauchy}
	Under Assumption \ref{assumption1}, the sequence $(\nu_n(\bp))_n$ is a uniformly equi\-continuous $\bar{f}$-Cauchy sequence in the following sense: for every $\varepsilon>0$  there is $n_0\ge1$ such that for all $k,\ell\ge n_0$ and all $\bp\in\cM_{\rm B}(\cA^\bZ,\sigma_\cA)$, 
\[
	\bar f\big(\nu_k(\bp),\nu_\ell(\bp)\big)
	<\varepsilon.	
\]
\end{theorem}

\begin{proof}
Let us sketch first the sequence of arguments to prove the assertion.
Given a Bernoulli vector $\bp$, let $\bp_n$ and $\widetilde{\bp_n}$ be the vectors defined by \eqref{conconj} and \eqref{conconjtilde}, respectively, and consider their Bernoulli-coded measure $\kappa$ as well as its invariant version $\kappa_{\rm inv}$ (Definition \ref{defBerncod}). Recall that Lemma \ref{lemfact:bar} implies that
\begin{equation}\label{factinv}
	\bar{f}\big( \kappa_{\rm inv}(\cA_n, \varrho_n, \widetilde{\bp_n}), 
			\kappa(\cA_n, \varrho_n, \widetilde{\bp_n})\big)=0.
\end{equation}
Therefore, below we will focus on the $\bar f$-distances between the measures $\kappa(\cdot)$ only.
Because of the ``lack of self-similarity'' of $\widetilde{\bp_n}$, we first prove the Cauchy property for the sequence $(\kappa(\cA_n, \varrho_n, \bp_n))_n$. Then we conclude our arguments using that, by Corollary \ref{corthisisit}, $\bp_n$ and $\widetilde{\,\bp_n\,}$ are ``asymptotically close'' for $n$ large. 

\begin{claim} \label{lemprop:later}
For any numbers $k,n\in\bN$, $k<n$, and any $\bp\in\cM_{\rm B}(\cA^\bZ,\sigma_\cA)$,
\[
	\bar{f}(  \kappa(\cA_k, \varrho_k, \bp_k),
 			\kappa(\cA_n, \varrho_n, \bp_n))
	\le 4K 2^{-k},
\]
where $K>0$ is as in  Assumption \ref{assumption1}.
\end{claim}

\begin{proof}
Note that $\cA_n=\cA^{m_1\cdots m_n}=(\cA_k)^{m_{k+1}\cdots m_n}$. Recall that
\[
	\Sub_{n,k}(\cA_n)
	=(\cA_k)^{m_{k+1}\cdots m_n}
\]
and that $\bp_k$ is Bernoulli (with probabilities given by Lemma \ref{lemend}). The same way, $\bp_n=\bp_{k+(n-k)}$ is Bernoulli and its probability vector can be written in terms of the products of the probabilities of the vector $\bp_k$. 

Given $a^{(n)}\in\cA_n$ and $(a^{(k)}_1,\ldots ,a^{(k)}_{m_{k+1}\cdots m_n})\eqdef \Sub_{n,k}(a^{(n)})\in(\cA_k)^{m_{k+1}\cdots m_n}$, let 
\begin{equation}\label{substring}
	\varrho'(a^{(n)})
	=\varrho_k(a^{(k)}_1)\ldots\varrho_k(a^{(k)}_{m_{k+1}\cdots m_n}).
\end{equation}
Applying Lemma \ref{lemend} to $\cA=\cA_k$, $\cA'=\cA_n=(\cA_k)^{m_{k+1}\cdots m_n}$ and noting that $(\bp_k)'=\bp_n$, we get
\[
	\kappa(\cA_k, \varrho_k, \bp_k)
	= \kappa\big(\cA_n, \varrho' , (\bp_k)'\big)
	= \kappa\big(\cA_n, \varrho' , \bp_n\big).
\]
It follows then from Lemma \ref{lemma:tailslength} that
\[
	(1- 4K2^{-k}) \big|\varrho_n(a^{(n)})\big|
	< \big|\varrho_k(a^{(k)}_1)\ldots\varrho_k(a^{(k)}_{m_{k+1}\cdots m_n})\big|.	
\]
Note that by Assumption \ref{assumption1}, \eqref{substring} is a substring of $\varrho_n(a^{(n)})$.
Hence, it suffices to apply Proposition \ref{prolem:Y} to obtain the assertion.
\end{proof}

%

We are now ready to conclude the proof.
Given $\varepsilon>0$, let $k_1\in\bN$ be large enough so that
\[
	4K2^{-k_1}<\frac\varepsilon3.
\]
Let $k_2$ be provided by Lemma \ref{lemprop:earlier} applied to $\varepsilon/3$. For any $k\ge\max\{k_1,k_2\}$ and $\ell>k$, 
\[\begin{split}
	\bar f\big(\nu_k(\bp),\nu_\ell(\bp)\big)
	&=\bar f\big(\kappa_{\rm inv}(\cA_k,\varrho_k,\widetilde{\,\bp_k\,}),
		\kappa_{\rm inv}(\cA_\ell,\varrho_\ell,\widetilde{\,\bp_\ell\,})\big)\\
	\text{\tiny{(by \eqref{factinv})}}\quad
	&= \bar f\big(\kappa(\cA_k,\varrho_k,\widetilde{\,\bp_k\,}),\kappa(\cA_\ell,\varrho_\ell,\widetilde{\,\bp_\ell\,})\big)\\
	&\le \bar f\big(\kappa(\cA_k,\varrho_k,\widetilde{\,\bp_k\,}),\kappa(\cA_k,\varrho_k,\bp_k)\big)\\
	&\phantom{=}+
	\bar f\big(\kappa(\cA_k,\varrho_k,\bp_k),\kappa(\cA_\ell,\varrho_\ell,\bp_\ell)\big)\\
	&\phantom{=}+
	\bar f\big(\kappa(\cA_\ell,\varrho_\ell,\bp_\ell),\kappa(\cA_\ell,\varrho_\ell,\widetilde{\,\bp_\ell\,}\big)\\
	\text{\tiny{(by Lemma \ref{lemprop:earlier} and Claim \ref{lemprop:later})}}\quad
	&< \frac\varepsilon3+4K2^{-k}+\frac\varepsilon3
	< \frac\varepsilon3+\frac\varepsilon3+\frac\varepsilon3=\varepsilon.
\end{split}\]
This finishes the proof of the theorem.
\end{proof}

The following is an immediate consequence of Lemma \ref{ergodicLB}, Remark \ref{remfbars} ($\bar f$-complete\-ness),  and Lemma \ref{lemLB} (``LB-completeness'').

\begin{corollary}\label{corerg}
	Under Assumption \ref{assumption1},  the sequence $(\nu_n(\bp))_n$ $\bar f$-converges to some probability measure $\nu_\infty(\bp)$. This limit measure $\nu_\infty(\bp)$ is LB (and hence ergodic) and 
\[
	h(\sigma_\cB,\nu_\infty(\bp))=\lim_{n\to\infty}h(\sigma_\cB,\nu_n(\bp)).
\]	
\end{corollary} 

\subsection{Properties of the space of all Bernoulli-coded measures}	\label{ssecFurtherStru}

The goal of this subsection is to combine all the results obtained above and to describe the topological properties of the space of Bernoulli-coded measures. For that, we fix a sequence of natural numbers $(m_n)_n$, a sequence of substitutions $(\varrho_n)_n$ from $\cA_n$ into the words over a common finite alphabet $\cB$, and assume that they satisfy Assumption \ref{assumption1}. 

We are going to study what happens when we vary the Bernoulli measure $\bp$.
Note that a map from $\cM_{\rm B}(\cA^\bZ,\sigma_\cA)$ into the space $\cM_{\rm erg}(\cB^\bZ,\sigma_\cB)$ is \emph{$\bar f$-continuous} if it is continuous in the $\bar f$-topology on the latter space, that is, if for any convergent sequence of Bernoulli measures their images converge in the $\bar f$-topology on $\cM_{\rm erg}(\cB^\bZ,\sigma_\cB)$.

Recall our notations
\[
	\nu_n(\bp)
	\eqdef \kappa_{\rm inv}(\cA_n,\varrho_n,\widetilde{\,\bp_n\,}),\quad
	\nu_\infty(\bp)
	\eqdef \lim_{n\to\infty}\nu_n(\bp).
\]

\begin{theorem} \label{theoprop:pathfinal}
Consider finite alphabets $\cA$ and $\cB$ and a cascade of repeat-and-tail substitutions $\varrho_n\colon\cA_n\to\cB^\ast$ satisfying Assumption \ref{assumption1}. Then the map 
\[
	\nu_\infty\colon\cM_{\rm B}(\cA^\bZ,\sigma_\cA)\to\cM_{\rm erg}(\cB^\bZ,\sigma_\cB)
\]
is well defined and $\bar f$-continuous. Any such limit measure is LB. Moreover, the map
\[
	\cM_{\rm B}(\cA^\bZ,\sigma_\cA)\ni\bp\mapsto h(\sigma_\cB,\nu_\infty(\bp))
\]
is continuous. In particular, $\nu_\infty(\cM_{\rm B}(\cA^\bZ,\sigma_\cA))$ is $\bar f$-path-connected and 
\[
	\big\{h(\sigma_\cB,\nu_\infty(\bp))\colon\bp\in\cM_{\rm B}(\cA^\bZ,\sigma_\cA)\big\}
\]
 is a closed interval.
\end{theorem}

The key argument towards the proof of the above theorem is the following.

\begin{lemma}\label{lemclaimArz}
The sequence of maps 
\[
	\nu_n\colon \cM_{\rm B}(\cA^\bZ,\sigma_\cA)\to\cM_{\rm erg}(\cB^\bZ,\sigma_\cB),\quad
\]
is uniformly equicontinuous in the following sense: for any $\varepsilon >0$ there exist $n_0\in\bN$ and $\delta>0$ such that for any two measures $\bp, \bp'\in \cM_{\rm B}(\cA^\bZ,\sigma_\cA)$ satisfying
\[
D_\cA(\bp,\bp') < \delta,
\]
for all $n\ge n_0$ we have 
\[
\bar f(\nu_n(\bp), \nu_n(\bp')) < \varepsilon.
\]
\end{lemma}

\begin{proof}
Given $\varepsilon>0$, let $n_0=n_0(\varepsilon)$ as provided by Theorem \ref{theCauchy}. For every $n\ge n_0$, $k=n_0$, we get
\[\begin{split}
	\bar f(\nu_n(\bp), \nu_n(\bp')) 
	&\leq \bar f (\nu_n(\bp), \nu_k(\bp))+ \bar f (\nu_k(\bp), \nu_k(\bp'))
		+ \bar f (\nu_k(\bp'), \nu_n(\bp'))\\
	{\tiny{\text{(by Theorem \ref{theCauchy})}}}\quad
	&\le \varepsilon +  \bar f (\nu_k(\bp), \nu_k(\bp'))+ 	\varepsilon.
\end{split}\]
We get the following estimates 
\[\begin{split}
	 \bar f (\nu_k(\bp), \nu_k(\bp')&)
	 =\bar{f}\big( \kappa_{\rm inv}(\cA_k,\varrho_k,\widetilde{\,\bp_k\,}), 
	 			\kappa_{\rm inv}(\cA_k,\varrho_k,\widetilde{\,\bp_k'\,})\big)\\
	 \text{\tiny{(by Lemma \ref{lemfact:bar})}}\quad
	 &=\bar{f}\big( \kappa(\cA_k,\varrho_k,\widetilde{\,\bp_k\,}), 
	 			\kappa(\cA_k,\varrho_k,\widetilde{\,\bp_k'\,})\big)\\
	 &\le\bar{f}\big( \kappa(\cA_k,\varrho_k,\widetilde{\,\bp_k\,}), 
	 			 \kappa(\cA_k,\varrho_k,\bp_k)\big)+\\
	&\phantom{=}				 
		+\bar{f}\big( \kappa(\cA_k,\varrho_k,\bp_k), 
	 			\kappa(\cA_k,\varrho_k,\bp_k')\big)+\\
	&\phantom{=}
		+\bar{f}\big( \kappa(\cA_k,\varrho_k,\bp_k'), 
	 			\kappa(\cA_k,\varrho_k,\widetilde{\,\bp_k'\,})\big)		\\
	 \text{\tiny{(by Lemma \ref{lemprop:earlier})}}\quad
	&\le \varepsilon+\bar{f}\big( \kappa(\cA_k,\varrho_k,\bp_k), 
	 			\kappa(\cA_k,\varrho_k,\bp_k')\big)+\varepsilon\\			
	 \text{\tiny{(by Proposition \ref{prop:P})}}\quad
	&\leq 2\varepsilon+ \frac {\max|\varrho_k|} {\min |\varrho_k|} \cdot D_{\cA_k}(\bp_k,\bp_k') \\
	\text{\tiny{(by Corollary \ref{cor:maxmin})}}\quad
	&\le 2\varepsilon+
	(1+4K)
		D_{\cA_k}(\bp_k,\bp_k').
\end{split}\]
Now apply Lemma \ref{lem:path1} to $(1+4K)
^{-1}\varepsilon$ and $k$ to get $\delta$ to conclude that if $D_\cA(\bp,\bp')<\delta$ then 
\[
	D_{\cA_k}(\bp_k,\bp_k')<
	(1+4K)^{-1}\varepsilon.
\]	
Hence, together, we get for every $n\ge n_0$ that
\[
	\bar f(\nu_n(\bp), \nu_n(\bp')) 
	\le 5\varepsilon.
\]
This implies the assertion.
\end{proof}

\begin{proof}[Proof of Theorem \ref{theoprop:pathfinal}]
By  Corollary \ref{corerg}, the map $\nu_\infty$ is well defined and the measures are LB. The continuity of this map is now a direct consequence of Lemma \ref{lemclaimArz} and the Arzel\`a-Ascoli theorem.

Finally observe that the set of all probability vectors $\{(p_a)_{a\in\cA}\colon p_a\ge0, \sum_ap_a=1\}$ is a simplex.  It follows from Lemma \ref{lemshields} that the space $\cM_{\rm B}(\cA^\bZ,\sigma_\cA)
$ is also a simplex in the topology generated by the $\bar f$-metric (as a subset of $\cM_{\rm erg}(\cA^\bZ,\sigma_\cA)$). By Lemma \ref{lementcont}, the map $\bp\mapsto h(\sigma_\cB,\nu_\infty(\bp))$ is continuous (in the $\bar f$-topology on $\cM_{\rm B}(\cA^\bZ,\sigma_\cA)$). In particular, the set of measures $\nu_\infty(\cM_{\rm B}(\cA^\bZ,\sigma_\cA))$ is $\bar f$-path connected and the entropies of all measures in $\nu_\infty(\cM_{\rm B}(\cA^\bZ,\sigma_\cA))$ form a closed interval. This proves the theorem.
\end{proof}

\section{Special cascades of horseshoes in circle diffeomorphisms}\label{secGeometry}

In this section, we consider the ``higher-dimensional'' context of step skew products with $\Sigma_N$ as base space and $C^1$ circle diffeomorphisms as fiber maps. We take this ``geometric setting'' and consider it from the point of view of Section \ref{secabscasc}. We will be largely following the constructions in \cite{DiaGelRam:22a}. The description of the similarities and crucial differences will be postponed to the beginning of Section \ref{sec52}. In Section \ref{secCIFSs}, we introduce contracting iterated function systems (CIFS) and the associated ``horseshoes''. In Section \ref{seccocyclediffeo}, we introduce the class $\mathrm{SP}^1_{\rm shyp}(\Sigma_N\times\bS^1)$. In Section \ref{secreptai}, we describe a cascade of CIFSs by ``repeating and tailing'' while essentially maintaining its ergodic properties and lowering its contraction rates. This cascade of associated alphabets and substitutions fits the setting of Section \ref{secabscasc}. Throughout this section we will prepare the proof of Theorem \ref{Bthm:circle} and complete it in Section \ref{secthmcircle}. 
  
Given $N\ge2$, consider a finite family $f_i\colon \bS^1\to \bS^1$, $i=1,\ldots,N$, of $C^1$ diffeomorphisms and the associated  step skew product $F$ defined as in \eqref{eq:sp}.
We will also consider the projection $\pi\colon\Sigma_N\times\bS^1\to\Sigma_N$, $\pi(\xi,x)\eqdef\xi$.
Given a measure $\mu$, consider its fiber Lyapunov exponent $\chi(F,\mu)$ defined as in \eqref{h0Fdef}.

For later reference, we state the following immediate consequence of \cite{LedWal:77}.

\begin{lemma}\label{entpres}
	For every $\mu\in\cM_{\rm erg}(\Sigma_N\times\bS^1,F)$, $h(F,\mu)=h(\sigma,\pi_\ast\mu)$.
\end{lemma}

\subsection{Collection of words giving rise to a CIFS and a horseshoe}\label{secCIFSs}

Let us first introduce some notation. Given $\xi=(\ldots,\xi_{-1}|\xi_0,\xi_1,\ldots)\in\Sigma_N$ and $n\in\bN$, write
\begin{equation}\label{eq:reference}
	 f_\xi^{-n}\eqdef f_{\xi_{-n}}^{-1}\circ\cdots \circ f_{\xi_{-1}}^{-1}
	 \quad\text{ and }\quad
	 f_\xi^n \eqdef f_{\xi_{n-1}}\circ\cdots\circ f_{\xi_0},
\end{equation}	

For $n\in\bN$ let $\Sigma_N^n\eqdef\{1,\ldots,N\}^n$ and define $\Sigma_N^\ast \eqdef \bigcup_{n=0}^\infty\Sigma_N^n$. Recall that the \emph{length} of a word $w\in\Sigma_N^\ast$ is the number of symbols it contains and is denoted by $| w|$. Given a finite subset $\cW$ of $\Sigma_N^\ast$ let
\begin{equation}\label{defNormW}
 	\lVert\cW\lVert
	\eqdef \max_{w\in\cW}|w|.
\end{equation}
 Analogously to notation \eqref{eq:reference}, given words $w_1,\ldots,w_m\in\Sigma_N^\ast$, $L\eqdef \lvert w_1\ldots w_m\rvert$, and for $k=1,\ldots, L$, denote by
\[
	f_{[w_1\ldots w_m]}^k
	\eqdef f_{\xi_{k-1}}\circ\cdots\circ f_{\xi_0},\,\text{where}\,
	(\xi_0,\xi_1,\ldots,\xi_{L-1})\eqdef(w_1\ldots w_m)
		\in\Sigma_N^\ast,
\]
the map obtained by concatenating the maps from the family $\{f_i\}_i$ which are indexed by the first  $k$  elements  of the concatenated words $w_1\ldots w_m$ (in the alphabet $\{1,\ldots,N\}$). 	Moreover, we simply write
\[
	f_{[w_1\ldots w_m]}
	\eqdef f_{[w_1 \ldots w_m]}^L.
\]

\begin{definition}[CIFS with quantifiers]
	A finite collection of words $\cW\subset\Sigma_N^\ast$ defines a \emph{contracting iterated function system} (\emph{CIFS}) on an interval  $J\subset\bS^1$ relative to $K>1$, $\alpha_0<0$, $\alpha<0$, and  $\varepsilon\in(0,\lvert\alpha\rvert)$ if
\begin{itemize}
\item[(a)] for every $ w\in\cW$ it holds $f_{[ w]}(J)\subset J$,
\item[(b)] for every $m\in\bN$, $ w_1,\ldots, w_m\in\cW$, $y\in J$, and $k=1,\ldots,\lvert w_1\ldots w_m\rvert$,
\[
	\lvert (f_{[ w_1\ldots w_m]}^k)'(y)\rvert
	\le Ke^{k\alpha_0},
\]
\item[(c)]
the \emph{spectrum of finite-time fiber Lyapunov exponents} satisfies
\[
	\Big\{\frac{1}{| w|}\log\,\lvert (f_{[ w]})'(x)\rvert
	\colon x\in J, w\in\cW\Big\}
	\subset (\alpha-\varepsilon,\alpha+\varepsilon).
\]
\end{itemize}
\end{definition}

For further reference, let us state a technical distortion result that does not require any further structure. Here we use the common notation 
\[
	S_n\phi\eqdef\phi+\phi\circ F+\ldots+\phi\circ F^{n-1}.
\]
Recall the notation of cylinders in \eqref{defcylinder}.	 

\begin{lemma}[{\cite[Proposition 6.12]{DiaGelRam:22a}}]\label{lemdistoprop}
	Let $\cW\subset\Sigma_N^\ast$ be a finite collection of words defining a CIFS on an interval $J\subset\bS^1$ relative to $K>1$, $\alpha_0<0$, $\alpha<0$, and  $\varepsilon$. Then for every continuous function $\phi\colon\Sigma_N\times\bS^1\to\bR$  and $\tau>0$, there exists $N_1=N_1(\phi,\tau)\in\bN$ such that for every $m\ge N_1$ and every concatenated word $w_1\ldots w_m\in\cW^m$, $ w_1,\ldots, w_m\in\cW$, 
\[
	\max_{(\xi,x),(\eta,y)\in[w_1\ldots w_m]\times J}
		\big|S_n\phi(x,\xi)-S_n\phi(y,\eta)\big|<\tau n,
		\text{ where }
		n\eqdef\sum_{i=1}^m|w_i|.
\]	
\end{lemma}

We state the following consequence of the \emph{contracting} property of the IFSs we study.

\begin{lemma} \label{lem:sbe}
Let $\cW\subset\Sigma_N^\ast$ be a finite collection of words defining a CIFS on an interval $J$. Then for every $\varepsilon>0$ there exists $N_2\in\bN$ such that for every $m\ge N_2$ and $w_1,\ldots,w_m \in \cW$,
\[
	\max_{(\xi,x),(\eta,y)\in[w_1\ldots w_m]\times J}
	\frac1n  \sum_{i=0}^{n-1} d(F^i(\xi,x), F^i(\eta,y)) \leq \varepsilon,
	\text{ where }
	n\eqdef\sum_{i=1}^m |w_i|.
\] 
\end{lemma}
\begin{proof}
Assume that $\cW$ gives rise to a CIFS relative to $K>1$, $\alpha_0<0$, $\alpha<0$, and $\varepsilon>0$. Then the uniform contraction implies that for every $m\in\bN$ and $k=1,\ldots,|w_1\ldots w_m|$,
\[
	\diam f^k_{[w_1\ldots w_m]}(J)
	\le Ke^{k\alpha_0}.
\]
This immediately implies the assertion.
\end{proof}

We need the following concept from \cite[Section 3.1]{DiaGelRam:22a}.

\begin{definition}[Disjoint collection of words]
	A collection of words $\cW	\subset\Sigma_N^\ast$ is \emph{disjoint} if no element in $\cW$ is a prefix of another element in $\cW$.
\end{definition}

\begin{remark}[Disjointness and decipherability]\label{remproblem}
By a slight abuse of notation and with the intention to simplify notation, in the following, when writing $\cW^\bZ$, we understand this as a subset of $\Sigma_N$ consisting of bi-infinite concatenations of words from $\cW$. This set is also called the \emph{pre-coded space} defined by $\cW$. The \emph{coded space} defined by $\cW$ is obtained by taking the closure of the union of the images of the pre-coded space under iterations by the shift map $\sigma$. See, for example, \cite[Chapter 13.5]{LinMar:95}.

In what is below, we always consider CIFSs defined by means of \emph{disjoint} collections of words. Note that every disjoint finite collection of words is \emph{uniquely left decipherable}, that is, whenever a concatenated word $w_1\ldots w_m$ is a prefix of another concatenated word $v_1\ldots v_n$, where $w_i,v_j\in\cW$, then $m\le n$ and $w_i=v_i$ for every $i=1,\ldots,m$. This left-decipherability property extends to one-sided infinite concatenations, that is, the one-sided concatenation space $\cW^\bN$ for disjoint $\cW$ is also uniquely decipherable. The bi-infinite sequences in the two-sided concatenation space  $\cW^\bZ$ for disjoint $\cW$ are in general not uniquely decipherable. We have however the result below (see \cite[Section 3.1]{DiaGelRam:22a} for further details).
\end{remark}

\begin{lemma}
	Let $\cW\subset\Sigma_N^\ast$ be a disjoint finite collection of words. Then every sequence in $\cW^\bZ$ has at most $r\eqdef\max_{w\in\cW}|w|$ ``decodings'', that is, it can be written as a bi-infinite concatenation of words from $\cW$ in at most $r$ ways. 
\end{lemma}

\begin{lemma}\label{lemcWconc}
	Let $\cW\subset\Sigma_N^\ast$ be a disjoint finite collection of words defining a CIFS on $J$ relative to $K,\alpha_0,\alpha$, and $\varepsilon$. Then for every $m\in\bN$, $\cW^m$ is also disjoint and defines a CIFS on $J$ relative to $K,\alpha_0,\alpha$, and $\varepsilon$. 
\end{lemma}

To any collection of words $\cW$ defining a CIFS, we can associate its ``attractor'' 
\[\begin{split}
	&\Lambda(\cW)\eqdef\Pi_\cW(\cW^\bZ)\subset\Sigma_N\times J,\quad\text{where}\\
	&\Pi_\cW\colon\cW^{\bZ}\to\Sigma_N\times J,\quad
	\Pi_\cW(\underline w)
	\eqdef\big(\underline w,\lim_{n\to\infty}f_{[w_{-1}]}\circ\cdots\circ f_{[w_{-n}]}(x_0)\big),
\end{split}\]
see \cite[Section 6]{DiaGelRam:22a}. Note that the limit above indeed does not depend on the choice of $x_0\in J$, see \cite{Hut:81}. Moreover, when $\cW$ is disjoint then $\Pi_\cW$ is uniformly finite-to-one and 
\[
	\card\Pi_\cW^{-1}(\{(\xi,x)\})\le \max_{w\in\cW}|w|
	\quad\text{ for all }\quad (\xi,x)\in\Sigma_N\times J,
\]
see \cite[Proposition 6.3]{DiaGelRam:22a}.

The attractor $\Lambda=\Lambda(\cW)$ defines the associated ``horseshoe''
\begin{equation}
\label{eq:horseshoe}
	\Gamma=\Gamma(\cW)\eqdef\bigcup_{k=0}^{\lVert\cW\rVert-1}F^k(\Lambda)\subset\Sigma_N\times\bS^1.
\end{equation}	
Note that the set $\Pi_\cW(\cW^\bZ)$ defined above can be seen as a ``section'' of $\Gamma(\cW)$, where each point in the horseshoe $\Gamma(\cW)$ hits it after at most $\max_{w\in\cW}|w|$ iterations by $F$.  

The following fact is straightforward. 

\begin{lemma}\label{lemWassm}
	For every disjoint finite collection $\cW\subset\Sigma_N^\ast$ defining a CIFS and its associated attractor $\Gamma(\cW)$ and every $m\in\bN$, 
\[
	\cM_{\rm erg}(\Gamma(\cW^m),F)
	= \cM_{\rm erg}(\Gamma(\cW),F).
\]	
\end{lemma}

\begin{remark}[Abstract alphabet $\cA$]\label{remabstract}
The collection of words $\cW=\{w_1,\ldots,w_M\}\subset\Sigma_N^\ast$ can itself be considered as an alphabet giving rise to a sequence space. In order to stress the distinction between this sequence space and the set of bi-infinite concatenated words $\cW^\bZ\subset\Sigma_N$ that we introduced above, let us call it 
by a different name and consider an abstract alphabet of letters $\cA=\{a_1,\ldots,a_M\}$ that has the same cardinality as $\cW$, $M=\card\cW$. We will use the associated sequence space $\cA^\bZ$ with the shift map $\sigma_\cA$.
\end{remark}

Below we describe the internal structure of the horseshoes in \eqref{eq:horseshoe}. Suspension spaces will turn out to be a convenient tool, precisely because of the problem with non-unique decipherability (Remark \ref{remproblem}). As the next proposition asserts, this horseshoe is a topological factor of the suspension space, where this factor map is \emph{a priori} not a bijection, but it has bounded multiplicity. 

For the following result, recall the definition of a discrete suspension space in Section \ref{secsusp}. Given $\cW$ and its associated abstract alphabet $\cA$, consider the associated roof function
\begin{equation}\label{rooofs}
	R\colon\cA\to\bN,\quad
	R( a_k)\eqdef |w_k|,
\end{equation}
where $a_k\in\cA$ is the symbol corresponding to $w_k\in\cW$.

\begin{proposition}[Topological horseshoe factor, {\cite[Proposition 8.4]{DiaGelRam:22a}}]\label{pro:semiconjtop}
	Let $\cW\subset\Sigma_N^\ast$ be a disjoint finite collection of words defining a CIFS, $\cA$ be the associated abstract alphabet, and $\cS=\cS_{\cA,R}$ be the discrete-time suspension space with suspension map $\Phi=\Phi_{\cA,R}$, and $\Gamma=\Gamma(\cW)$ the associated horseshoe. 
	There is a continuous surjective map $H\colon \cS\to\Gamma$ that is uniformly finite-to-one such that
\[
	\card H^{-1}(\{(\xi,x)\})
	\le \left(\max R\right)^2,
	\spac{for every} (\xi,x)\in\Gamma,
\]	
satisfying $H\circ \Phi=F\circ H$.
\end{proposition}

Let us also consider associated invariant measures and their factors.

\begin{proposition}[Measure-preserving factor]\label{pro:semiconj}
	Assume the hypotheses of Proposition \ref{pro:semiconjtop}. Consider some $\sigma_\cA$-ergodic measure $\nu$, let $\lambda=\lambda_{\cA,R,\nu}$ be its  suspension. Letting
\begin{equation}\label{eq:defmun}	
	\mu
	\eqdef (H)_\ast\lambda,
\end{equation}
the measure preserving system $(\Gamma,F,\mu)$ is a factor of the measure preserving system $(\cS, \Phi,\lambda)$ by the map $H$ provided by Proposition \ref{pro:semiconjtop}. Moreover, $(\Gamma,F,\mu)$ is ergodic and  
$$
h(F,\mu) = h(\Phi,\lambda).
$$
\end{proposition}

\begin{proof}
As, by Proposition \ref{pro:semiconjtop}, $H$ is finite-to-one. Hence, \cite{LedWal:77} implies that%
\footnote{Here $h_{\rm top}(\Phi,A)$ denotes the topological entropy of $\Phi$ on $A$, see \cite{Bow:73}.}
\[
	\sup_{\lambda\colon(H)_\ast\lambda=\mu}h(\Phi,\lambda)
	= h(F,\mu)
		+\int_{\cS}h_{\rm top}(\Phi,H^{-1}(\{X\}))\,d\mu(X)
	= h(F,\mu)	.
\]	
As the entropy of a factor system is always smaller than or equal to the entropy of its extension, this implies equality of entropies. Ergodicity is an immediate consequence of the factor property.
\end{proof}

The diagram in Figure \ref{diagram} summarizes all topological constructions in this subsection. Given a disjoint collection of words $\cW$ and the corresponding CIFS and horseshoe together with the associated abstract alphabet $\cA$, some Bernoulli measure $\bp$, and the roof function $R$ as in \eqref{rooofs}, we have the following picture:
\begin{figure}[h]
\xymatrix{
	&(\cA^\bZ,\sigma_\cA,\bp)\hspace{-0.8cm}
	&\ar[d]
	(\cA)^\bZ\times\{0\}\ar[d]_{H}\subset\hspace{-1cm}
	&\ar[d]_{H}
	\fS_{\cA,R}		
	\ar[r]^{\Phi}
	&\fS_{\cA,R}
	\ar[d]^{H}\hspace{-0.7cm}
	&\hspace{-0.5cm}(\fS_{\cA,R},\Phi_{\cA,R},\lambda_{\cA,R,\bp})
	\ar[d]^{H_\ast}
	\\
	&
	&\Lambda(\cW)\subset
	&\Gamma(\cW)\ar[r]^{F}
	&\Gamma(\cW)&(\Gamma(\cW),F,\mu_{\cA,R,\bp})
}	
\caption{Suspension spaces and horseshoes}
\label{diagram}
\end{figure}

\subsection{The class $\mathrm{SP}^1_{\rm shyp}(\Sigma_N\times\bS^1)$}\label{seccocyclediffeo}

Let us now present the class of maps that we are investigating in what is below. Given $x\in\bS^1$, consider its \emph{forward} and \emph{backward orbits} defined by
\[
	\cO^+(x)
	\eqdef \bigcup_{n\ge0}\,\,\bigcup_{\xi\in\Sigma_N}f_{\xi}^n(x)
	\spac{and}
	\cO^-(x)
	\eqdef \bigcup_{m\ge1}\,\,\bigcup_{\xi\in\Sigma_N}f_{\xi}^{-m}(x),
\]
respectively. Analogously, forward/backward orbits $\cO^\pm(S)$ of a set $S$ are defined. 

\begin{definition}[The set $\mathrm{SP}^1_{\rm shyp}(\Sigma_N\times\bS^1)$]
A skew product $F$ as in \eqref{eq:sp} belongs to $\mathrm{SP}^1_{\rm shyp}(\Sigma_N\times\bS^1)$, $N\ge 2$, if the following properties hold: 

\medskip\noindent
\textbf{T (Transitivity).} $F$ is transitive (in the sense that there is $x\in\bS^1$ such that $\cO^+(x)$ and $\cO^-(x)$ are both dense in $\bS^1$). \smallskip

Moreover, we assume that there exists a closed nontrivial interval $J\subset\bS^1$, called a \emph{blending interval}, such that the following holds: 

\medskip\noindent
\textbf{ACC($J$) (Accessibility).} $J$ is backward and forward accessible: $\cO^\pm(\interior J)=\bS^1$.
\smallskip

Denote by $|I|$ the length of an interval $I\subset\bS^1$.

\medskip\noindent
\textbf{CEC+($J$) (Controlled Expanding forward Covering).}
There exist positive constants $K_1,\ldots,K_5$ so that for every interval $I\subset\bS^1$ intersecting $J$, $\lvert I\rvert<K_1$, there are $\ell\in\bN$ and a word $(\eta_0,\ldots,\eta_{\ell-1})\in\Sigma_N^\ast$ with  $\ell \le  K_2\,\lvert\log\,\lvert I\rvert\rvert +K_3$ such that
\begin{itemize}
\item  (controlled covering)  
\[
	\left(f_{\eta_{\ell-1}}\circ\cdots\circ f_{\eta_0}\right)(I)\supset B(J,K_4),
\]	
where $B(J,\delta)$ is the $\delta$-neighborhood of the set $J$,
\item  (controlled expansion) for every $x\in I$ it holds
\[
	\log \,\lvert  \left(f_{\eta_{\ell-1}}\circ\cdots\circ f_{\eta_0}\right)'(x)\rvert
	\ge \ell K_5.
\]	
\end{itemize}
\medskip\noindent
\textbf{CEC$-(J$) (Controlled Expanding backward Covering).} The step skew product  $F^{-1}$ satisfies the Axiom CEC$+(J)$.
\end{definition}

We refrain from providing an in-depth discussion of the class   $\mathrm{SP}^1_{\rm shyp}(\Sigma_N\times\bS^1)$, its definition is a translation of the properties of a large class of robustly transitive diffeomorphism to the skew product setting. We refer to \cite[Section 8.3]{DiaGelRam:17} for details, where it is observed that (besides its intrinsic interest)  this class is very well suited for studying nonhyperbolic  transitive diffeomorphisms. This observation was further developed in  \cite{DiaGelSan:20,YanZha:20}.  
The importance of this class in the context of elliptic cocycles is illustrated in \cite{DiaGelRam:19,DiaGelRam:22a}. Finally, for a discussion of the interplay between transitivity and ``minimality of foliations'', see \cite{BarCis:}.

In what is below, we assume that $F\in\mathrm{SP}^1_{\rm shyp}(\Sigma_N\times\bS^1)$ and that $J\subset\bS^1$ is some associated blending interval. We will omit any further mention of them.

Consider the \emph{Wasserstein distance} on $\cM(\Sigma_N\times\bS^1,F)$,
\begin{equation}\label{wasss}
	W(\mu,\mu')
	\eqdef \sup\Big\{\Big|\int\phi\,d\mu-\int\phi\,d\mu'\Big|\colon\phi\in\Lip(1)\Big\},
\end{equation}
where $\Lip(1)$ denotes the space of all Lipschitz continuous functions $\phi\colon\Sigma_N\times\bS^1\to\bR$ whose Lipschitz constant is $\le1$. Recall that the Wasserstein distance is indeed a metric and induces the weak$\ast$ topology.

\begin{remark}[Determining the Wasserstein distance]\label{remWass}
Note that the formula \eqref{wasss} does not change if we add a constant to the gauge function $\phi$. Thus, it is enough to restrict ourselves to  the space $\Lip_0(1) \subset \Lip(1)$ of non-negative $1$-Lipschitz functions whose values are bounded by $\diam(\Sigma_N\times\bS^1)$, that is,
\begin{equation}\label{wasssb}
	W(\mu,\mu')
	= \sup\Big\{\Big|\int\phi\,d\mu-\int\phi\,d\mu'\Big|\colon\phi\in\Lip_0(1)\Big\}.
\end{equation}
\end{remark}

The following result guarantees the existence of CIFSs discussed in Section \ref{secCIFSs}. 
Recall notation \eqref{defNormW}. Recall that we assume $F\in\mathrm{SP}^1_{\rm shyp}(\Sigma_N\times\bS^1)$.

\begin{proposition}[Existence of a CIFS with quantifiers]\label{proteo:existenceCIFS}
	Let $\theta$ be an $F$-ergodic measure with Lyapunov exponent $\alpha=\chi(F,\theta)<0$ and entropy $h=h(F,\theta)>0$. For every $\varepsilon_E\in(0,\lvert\alpha\rvert/4)$, $\varepsilon_H\in(0,h)$, and $\varepsilon_W>0$ there exists a  disjoint finite collection of words $\cW\subset\Sigma_N^\ast$ that defines a CIFS on $J$ relative to some constant $K>1$ and positive numbers $\alpha+\varepsilon_E$, $\alpha$, and $\varepsilon_E$ such that
\begin{equation}\label{eqass}
	\Big|\frac{1}{\lVert\cW\rVert}\log\card\cW-h\Big|
	\le \varepsilon_H.
\end{equation}
Moreover, 
\begin{equation}\label{eqassb}
	W(\mu,\theta)<\varepsilon_W
	\quad\text{ for all }\quad \mu\in\cM_{\rm erg}(\Gamma(\cW),F).
\end{equation}
Furthermore, $\cW$ can be chosen such that $|w|$ is constant in $\cW$.
\end{proposition}

\begin{proof}
By \cite[Theorem 6.5]{DiaGelRam:22a}, there exists a disjoint collection $\cW$ defining a CIFS as claimed and satisfying
\begin{equation}\label{eqassc}
	\Big|\frac{1}{\lVert\cW\rVert}\log\card\cW-h\Big|
	\le \varepsilon_H.
\end{equation}
 Let us argue that also \eqref{eqassb} is satisfied. Indeed, the proof of \cite[Theorem 6.5]{DiaGelRam:22a} is essentially based on \cite[Claim 6.6 (Existence of skeletons)]{DiaGelRam:22a} and \cite[Proposition 4.11]{DiaGelRam:17}. Given $\varepsilon_W>0$, there exists a finite collection of continuous functions $\phi_1,\ldots,\phi_{i(\varepsilon_W)}$ such that any function in $\Lip_0(1)$ is $\varepsilon_W$-close to one of them. Hence, it follows from Remark \ref{remWass} that for any $\mu\in\cM_{\rm erg}(\Gamma(\cW),F)$,
\[
	W(\mu,\theta)
	\le 2\varepsilon_W+\max_{k=1,\ldots,i(\varepsilon)}\Big|\int\phi_k\,d\mu-\int\phi_k\,d\theta\Big|.
\]
Then \cite[Proposition 4.11]{DiaGelRam:17} guarantees that $\cW$ can be chosen such that for any $\mu\in\cM_{\rm erg}(\Gamma(\cW),F)$ the righthand side in the above inequality is $\le3\varepsilon_W$. This implies the assertion \eqref{eqassb}.

What remains to show is that $\cW$ can be chosen to consist of words of equal lengths. Note that $\lVert\cW^m\rVert=m\lVert\cW\rVert$ and $\card(\cW^m)=(\card\cW)^m$. Choose $m\in\bN$ sufficiently large such that 
\begin{equation}\label{choicem}
	\frac{1}{m\lVert\cW\rVert}\log(m\lVert\cW\rVert)
	<\varepsilon_H.
\end{equation}	
Choose $k_0\in\bN$ so that 
\[
	\card\cV_{k_0}\eqdef\max_k\card\cV_k,\quad\text{ where }\quad
	\cV_k\eqdef\{w\in\cW^m\colon|w|=k\}.
\]	 
By \eqref{eqassc},
\[\begin{split}
	h-\varepsilon_{\rm H}
	&\le \frac{1}{\lVert\cW\rVert}\log\card\cW
	= \frac{1}{\lVert\cW^m\rVert}\log\card(\cW^m)\\
	&\le \frac{1}{\lVert\cW^m\rVert}\log\big(\lVert\cW^m\rVert\card\cV_{k_0}\big) \\
	&\le \frac{1}{m\lVert\cW\rVert}\log(m\lVert\cW\rVert)
		+\frac{1}{\lVert\cW^m\rVert}\log\card\cV_{k_0}\\
	\text{\tiny{(with \eqref{choicem})}}\quad	
	&\le 	\varepsilon_H+\frac{1}{\lVert\cV_{k_0}\rVert}\log\card\cV_{k_0}.
\end{split}\]
By Lemma \ref{lemcWconc}, $\cV\eqdef\cV_{k_0}$ is disjoint and defines a CIFS with the same quantifiers such as the CIFS defined by $\cW$. By the above estimates, assertion \eqref{eqass} is satisfied for $\cV$. The assertion \eqref{eqassb} for $\cV$ is a consequence of Lemma \ref{lemWassm}. 
\end{proof}

\subsection{Cascade of CIFSs, horseshoes, and substitutions}\label{secreptai}

We describe now a method to modify a given CIFS by ``repeating and tailing''.
These constructions are made possible by the fact that   $F\in\mathrm{SP}^1_{\rm shyp}(\Sigma_N\times\bS^1)$.

\begin{definition}[Repeat and tail]
Let $\cW\subset\Sigma_N^\ast$ be a disjoint finite collection of (nonempty) words and a number $m\in\bN$ (number of repetitions). Let $\ft=\ft_{\cW,m}\colon \cW^m\to\Sigma_N^\ast$ be a \emph{tailing map}. The new collection of words over the alphabet $\{1,\ldots,N\}$ obtained by \emph{$m$-times repeating and $\bt$-tailing} is given by
\[
	(\cW^m)_\ft
	\eqdef \big\{w_1\ldots w_m\ft(w_1 \ldots w_m)\colon w_k\in\cW \text{ for }k=1,\ldots,m\big\}
	\subset\Sigma_N^\ast .
\]
\end{definition}

Observe that \cite[Corollaries 3.5 and 3.6]{DiaGelRam:22a} imply the following result.

\begin{lemma}
	For any disjoint finite collection of words $\cW\subset\Sigma_N^\ast$, $m\in\bN$, and tailing map $\ft\colon \cW^m\to\Sigma_N^\ast$, the collection $(\cW^m)_\bt$ is disjoint. 
\end{lemma}

\begin{proposition}[Choice of a tailing map, {\cite[Theorem 7.3]{DiaGelRam:22a}}]\label{pro:tailing}
	There exists $L_1=L_1(F,J)>0$ such that the following holds.
	For every finite disjoint collection of words $\cW\subset\Sigma_N^\ast$ defining a CIFS on $J$ relative to  $K\in\bN,\alpha_0=\alpha+\varepsilon<0$, $\alpha<0$, and $\varepsilon\in(0,\lvert\alpha\rvert/2)$, there is $N_2=N_2(\cW)\in\bN$ such that for every $m\ge N_2$ there exists a tailing map $\ft=\ft_{\cW,m}\colon\cW^m\to\Sigma_N^\ast$ such that the $m$-times repeated and $\ft$-tailed collection of words $(\cW^m)_\ft$ defines a CIFS on $J$ relative to  $K,\alpha_0'$,  $\alpha'$, and $\varepsilon'$, where
\[
	\alpha_0'=\frac12(\alpha+\varepsilon),\quad
	\alpha'=\frac12\alpha,\quad
	\varepsilon'=\frac\varepsilon2.
\]
Moreover, for every $w_1\ldots w_m\in\cW^m$,
\begin{equation}\label{formulaa}
	\lvert\ft(w_1\ldots w_m)\rvert
	\le L_1\lvert\alpha\rvert\sum_{k=1}^m\lvert w_k\rvert .
\end{equation}
\end{proposition}

Assume that $\cW\subset\Sigma_N^\ast$ is a disjoint finite collection of (nonempty) words as provided by Proposition \ref{proteo:existenceCIFS}. Fix any sequence $(m_n)_n$ of natural numbers. Assume that $(m_n)_n$ grows fast enough in order to be able to apply Proposition \ref{pro:tailing} inductively. 

\begin{remark}[Fast growing sequences]\label{remfast}
Let us mention that the term ``grows fast enough'' means that each $m_n$ must be greater than some $M_n(m_1,\ldots, m_{n-1})$. Thus, we can always increase some $m_n$, but we might need to increase the following $m_{n+1}, m_{n+2},\ldots$ as well.
\end{remark}

We get a cascade of disjoint collections of words,
\[
	\cW_0\eqdef\cW,\quad
	\cW_1\eqdef (\cW_0^{m_1})_{\ft_{\cW_0,m_1}},\ldots,\quad
	\cW_n\eqdef(\cW_{n-1}^{m_n})_{\ft_{\cW_{n-1},m_n}},\ldots	.
\] 
Let $\cA$ be the associated abstract alphabet having the same cardinality as $\cW$ (recall Remark \ref{remabstract}). Consider the associated cascade of abstract alphabets $(\cA_n)_n$ as in \eqref{defalphscA}, where
\[
	\cA_0
	\eqdef \cA,\quad
	\cA_n\eqdef(\cA_{n-1})^{m_n}.
\]
To match the notation in Section \ref{secabscasc}, and using the informal identification of letters in $\cA_n$ with elements of $\cW_n$ in the spirit of Remark \ref{remabstract}, let us denote by $\bt_n$ the tailing function on the associated abstract alphabet $\cA_n$,
\begin{equation}\label{tailingmap}
	\bt_n\colon\cA_n\to\Sigma_N^\ast,\quad
	\bt_n(a^{(n)})
	\eqdef \bt_{\ft_{\cW_{n-1},m_n}}(w_1^{(n-1)}\ldots w_{m_n}^{(n-1)}).
\end{equation}

As for \eqref{rooofs}, consider the associated roof functions
\[
	R_n\colon\cA_n\to\bN,\quad
	R_n(a_k^{(n)})
	\eqdef |w_k^{(n)}|.
\]

We now look at associated measures. Fix any Bernoulli measure $\bp$ on $\cA^\bZ$. Notice that the following construction depends on the sequence $(m_n)_n$ and on $\bp$, though for simplicity this will not be reflected in our notation. Recalling the definition of the substitution maps $\underline\Sub_{n,0}$ in \eqref{eq:tailaddn-1n}, define a cascade of measures by letting $\bp_0\eqdef\bp$ and
\[
	\bp_n
	\eqdef (\underline\Sub_{n,0}^{-1})_\ast\bp_0,
\]
as in \eqref{conconj}.
Lemma \ref{newlemconjBer} implies the following.

\begin{lemma}\label{lemsubSubn0}
$
	\underline\Sub_{n,0}\colon
	\big((\cA_n)^\bZ,\sigma_{n},\bp_n\big)\to
	\big(\cA^\bZ,\sigma_\cA^{m_1\cdots m_n},\bp\big)
$
is a metric isomorphism.
\end{lemma}

The above gives rise to a corresponding cascade of diagrams as in Figure \ref{diagram}.
 In particular, we obtain a cascade of suspension spaces $\fS_n$ and measures $\lambda_n(\bp)$, then the attractors of a CIFS and associated $F$-invariant horseshoes $\Gamma_n$ and its factor maps $H_n$, and then a cascade of $F$-ergodic measures $\mu_n(\bp)$ defined as in \eqref{eq:defmun} supported on $\Gamma_n$. Recall the projection $\pi\colon\Sigma_N\times\bS^1\to\Sigma_N$, $\pi(\xi,x)\eqdef\xi$ and let
\[
	\Xi_n
	\eqdef \pi(\Gamma_n)
	\subset \Sigma_N.
\]
Recall our definition of $\nu_n(\bp)$ in \eqref{defnunp}. Let $\Pi_n=\Pi_{\varrho_n}$ be defined as in \eqref{defPi}. By Lemma \ref{lema34b} together with $\Pi_n=\pi\circ H_n$, we get
\begin{equation}\label{someformula}
	\nu_n(\bp)
	= \kappa_{\rm inv}(\cA_n,\varrho_n,\widetilde{\,\bp_n\,})
	= (\Pi_n)_\ast\lambda_n(\bp)
	= (\pi\circ H_n)_\ast \lambda_n(\bp).
\end{equation}	
Observe that $\Xi_n\subset\Sigma_N$ is a subshift and that $(\Xi_n,\sigma,\nu_n(\bp))$ is an ergodic automorphism. The diagram in Figure \ref{diagram2} summarizes the interrelations between the objects considered in this subsection:
\begin{figure}[h]
\hspace{-2cm}\xymatrix{
	&(\cA^\bZ,\sigma_\cA^{m_1\cdots m_n},\bp)\hspace{-0.7cm}
	&(\cA_n)^\bZ\times\{0\}\subset\hspace{-0.9cm}
	&\ar[d]_{H_n}\fS_n\ar[r]^{\Phi}
	&\fS_n\ar[d]^{H_n}\hspace{-0.2cm}
	&(\fS_n,\Phi_n,\lambda_n(\bp))\ar[d]^{(H_n)_\ast}
	\\
	&
	&\Sigma_N\times\bS^1\supset\hspace{-0.9cm}
	&\Gamma_n\ar[r]^{F}\ar[d]_{\pi}
	&\Gamma_n\ar[d]^{\pi}
	&(\Gamma_n,F,\mu_n(\bp))\ar[d]^{\pi_\ast}
	\\
	&
	&\Sigma_N\supset\hspace{-1.7cm}
	&\Xi_n\ar[r]^{\sigma}
	&\Xi_n
	&(\Xi_n,\sigma,\nu_n(\bp))
	}	
\caption{Cascade of suspension spaces, horseshoes, and their projected subshifts}
\label{diagram2}
\end{figure}

We now estimate the entropy of the ergodic automorphism $(\Xi_n,\sigma,\nu_n(\bp))$ and the fiber Lyapunov exponents on the cascade of horseshoes.

\begin{proposition}\label{Corentest}
Assume that $(\cW_n)_n\subset\Sigma_N^\ast$ gives the cascade of CIFSs defined above.
Then for every $n\in\bN_0$, we have
\[
	\big\{\chi(F,\widetilde\mu)\colon \widetilde\mu\in\cM_{\rm erg}(\Gamma_n,F|_{\Gamma_n})\big\}
	\subset\big(\frac{\alpha-\varepsilon}{2^n},
			\frac{\alpha+\varepsilon}{2^n}\big)
\]
and
\[
	h(F,\mu_n(\bp))
	=h(\sigma,\nu_n(\bp))
	\ge \frac{e^{-L_1\lvert\alpha\rvert}}{\lVert\cW\rVert} h(\sigma_\cA,\bp) .
\]	
\end{proposition}

\begin{proof}
The estimates of the Lyapunov exponents is an immediate consequence of the cascade of CIFSs provided by Proposition \ref{pro:tailing}.

To estimate the entropy, first note that $\sigma_\cA$ is the shift in the original alphabet $\cA$. In particular, Lemma \ref{lemsubSubn0} 
implies
\begin{equation}\label{eqentropy}
	h(\sigma_n,\bp_n)
	= h(\sigma_\cA^{m_1\cdots m_n},\bp)
	= m_1\cdots m_n h(\sigma_\cA,\bp)
\end{equation}
By Proposition \ref{pro:semiconj} together with Abramov's formula \eqref{eqAbramov} and the fact that $R_n$ is piecewise constant, we get
\[
	h(F,\mu_n(\bp))
	= h(\Phi_n,\lambda_n(\bp))
	=\frac{h(\sigma_n,\bp_n)}{\sum_{a\in\cA_n} R_n(a)\bp_n([a])}.
\]
Hence, together with \eqref{eqentropy}, we get	
\[\begin{split}
	h(F,\mu_n(\bp))
	&\ge\frac{m_1\cdots m_nh(\sigma_\cA,\bp)}{\max_{a\in\cA_n}R_n(a)}\\
	{\tiny\text{by  \eqref{formulaa}}}\quad
	&\ge \frac{m_1\cdots m_nh(\sigma_\cA,\bp)}
			{m_n(1+L_12^{-(n-1)}\lvert\alpha\rvert)\max R_{n-1}}\\
	&\ge\ldots
	\ge \frac{h(\sigma_\cA,\bp)}{\prod_{k=0}^{n-1}(1+L_12^{-k}\lvert\alpha\rvert)
		 \max R_0}\\
	&\ge \frac{e^{-L_1(1-2^{-n})\lvert\alpha\rvert}}{\max R_0}
		h(\sigma_\cA,\bp)\\
	{\text{\tiny{(together with \eqref{defNormW} and \eqref{rooofs})}}}\quad	
	&>	\frac{e^{-L_1\lvert\alpha\rvert}}{\lVert \cW\rVert}
		h(\sigma_\cA,\bp).
\end{split}\]
This, implies the assertion.
\end{proof}

\section{Limit measures on $\Sigma_N\times\bS^1$}\label{secGeometry-2}

Throughout this section, we continue to work in the setting of Section \ref{secGeometry}. 
In Section \ref{veryshortsec}, we collect our main ingredients and summarize in Section \ref{secscheme}  our general scheme. 

\subsection{Setting and collection of main ingredients}\label{veryshortsec}

We have given (compare Figure \ref{diagram2}):
\begin{enumerate}
\item[(i)] a map $F\in\mathrm{SP}^1_{\rm shyp}(\Sigma_N\times\bS^1)$ and a corresponding blending interval $J$,
\item[(ii)] an $F$-ergodic measure $\theta$ with negative fiber Lyapunov exponent $\alpha=\chi(F,\theta)$ and positive entropy $h=h(F,\theta)$,
\item[(iii)] a disjoint finite collection $\cW\subset\Sigma_N^\ast$ of words of equal lengths associated to $\theta$, provided by Proposition \ref{proteo:existenceCIFS} and  satisfying
\[
	\frac{1}{\lVert\cW\rVert}\log\card\cW\approx h
\] together with the associated abstract alphabet $\cA$ (recall Remark \ref{remabstract}), moreover
\[
	W(\mu,\theta)\approx 0
	\quad\text{ for all }\quad\mu\in\cM_{\rm erg}(\Gamma(\cW),F),
\]
\item[(iv)] a sequence of natural numbers $(m_n)_{n\in\bN}$ (numbers of repetition) that is sufficiently fast growing that allows us to repeatedly apply Proposition \ref{pro:tailing} (for some tailing maps $\bt_n\colon\cA_n\to\Sigma_N^\ast$ satisfying Assumption \ref{assumption1} with $K=L_1\lvert\alpha\rvert$) to obtain a sequence of substitution maps $\varrho_n\colon \cA_n\to\Sigma_N^\ast$,
as $\cW$ is a collection of words of of equal lengths, in particular $\max|\varrho_0|=\min|\varrho_0|$,
\item[(v)] resulting cascade of collections of words $(\cW_n)_n$, abstract alphabets $(\cA_n)_n$, suspension spaces $\fS_n$, horseshoes $\Gamma_n$, and subshifts $\Xi_n$.
\item[(vi)] for any Bernoulli measure $\bp\in \cM_{\rm B}(\cA^\bZ,\sigma_\cA)$ we get a sequence of ergodic measures $\lambda_n=\lambda_n(\bp)$ on $\fS_n$, $\mu_n=\mu_n(\bp)$ on $\Gamma_n$, and $\nu_n=\nu_n(\bp)$ on $\Xi_n$.
\end{enumerate}

The topological part (i)--(v) (construction of a cascade of horseshoes) follows \cite{DiaGelRam:22a}. However, there only the maximal entropy Bernoulli measure on $\cA^\bZ$ is considered. In the present paper, we need to consider \emph{arbitrary Bernoulli measures simultaneously}, giving rise to item (vi). 
This is why we revisit  the assertions in \cite{DiaGelRam:22a} in greater details. The crucial point, that we need to prove, is that we can choose {\it one} sequence $(m_n)_{n\in \bN}$ that works for \emph{all} Bernoulli measures simultaneously. For that, certain elements of the constructions in \cite{DiaGelRam:22a} must be made uniform. As a  final result of this improved construction, we  obtain a whole family of ergodic  zero fiber exponent measures, whose entropy varies continuously between 0 and $h(F,\theta)-\varepsilon$, for an arbitrarily small $\varepsilon$.

\begin{remark}[Trivial Bernoulli measures]\label{remsimples}
	Given $b\in \cA$ and taking any trivial probability vector $\bp=(p_a)_{a\in\cA}$ with $p_a=1$ for  $a=b$ produces zero fiber exponent ergodic measures with zero entropy. In this case, our construction is essentially identical with the ones in \cite{GorIlyKleNal:05,KwiLac:}. Oversimplifying, the more general case of our construction ``just'' replaces periodic orbits by horseshoes. 
\end{remark}

Some ingredients from Sections \ref{secabscasc} and \ref{secGeometry} are recalled in Figure \ref{diagram3bleft}. The following is an important property of our construction, see \eqref{someformula}.

\begin{claim}\label{claimfinal}
The  diagram in Figure \ref{diagram3bleft} commutes.
\end{claim}


\begin{figure}[h]
	\hspace{0.5cm}\xymatrix{
	(\cA^\bZ,\sigma_\cA,\bp)\ar[r]^{\eqref{susnu}} \ar[rd] \ar[rdd]_{\bp\mapsto\kappa_{\rm inv}(\cA_n,\varrho_n,\widetilde{\,\bp_n\,})} 
								&\lambda_n(\bp)\ar[d]^{(H_n)_\ast}^{(H_n)_\ast} \ar@/^3pc/[dd]^{(\Pi_n)_\ast}
								&
								&
								\\
								&\mu_n(\bp)\ar[d]^{\pi_\ast}
								&
								&\\
								&\nu_n(\bp) 
								&
								&
	}
\caption{Measures in the cascade}	
\label{diagram3bleft}
\end{figure}

\subsection{Road map to convergence results}\label{secscheme}

By Theorem \ref{theoprop:pathfinal}, we know already that the $\bar f$-limit $\nu_\infty(\bp)=\lim_n\nu_n(\bp)$ is well defined and LB (and hence ergodic). However, this results is purely symbolic and does not carry any information about fiber Lyapunov exponents. This information will come from the weak$\ast$ limit $\mu_\infty(\bp)=\lim_n\mu_n(\bp)$, we will show that this limit measure is indeed well defined, $F$-ergodic,  and has exponent zero; see Section \ref{sec52}. The weak$\ast$ continuity of $\bp\mapsto  \mu_\infty(\bp)$ explored in Section \ref{sec52b} will provide path-connectedness of the measures in Theorem \ref{Bthm:circle}. The continuity of the entropy map $\bp\mapsto h(F,\mu_\infty(\bp))$ will be a consequence of the continuity of $\bp\mapsto h(\sigma,\nu_\infty(\bp))$ (again Theorem \ref{theoprop:pathfinal}). 

\begin{figure}[h]
	\hspace{0.5cm}\xymatrix{
								\mu_n(\bp) \ar[r]^{{\rm weak}\ast} \ar[d]_{\pi_\ast} 
								&\mu_\infty(\bp) \ar[d]^{\pi_\ast}\\
								\nu_n(\bp) \ar[r]^{\bar f}
								&\nu_\infty(\bp)
	}
\caption{Limit measures in the cascade}	
\label{diagram3bright}
\end{figure}

\subsection{Existence and ergodicity of limit measures}	\label{sec52}

Note that it follows already from Proposition \ref{Corentest}, that any weak$\ast$ limit of the sequence of $F$-ergodic measures $(\mu_n(\bp))_n$ is a probability measure with fiber Lyapunov exponent zero. Ergodicity of a measure in the space $\Sigma_N\times\bS^1$ does not follow from the ergodicity of its factor in $\Sigma_N$ and must be proven separately.
Choosing some \emph{sufficiently} fast-growing sequence of positive integers $(m_n)_{n\in\bN}$ guarantees that the sequence $(\mu_n(\bp))_n$ in fact converges and that the limit measure is $F$-ergodic. 

Below we state a strengthened version of \cite[Proposition 10.1]{DiaGelRam:22a} and sketch its proof.
The only, but crucial, difference of the new statement compared with \cite[Proposition 10.1]{DiaGelRam:22a}
is the existence of a sequence $(m_n)_n$ for which the assertion holds true for not only for the $\sigma_\cA$-maximal entropy measure but for \emph{all} Bernoulli measures $\bp$, simultaneously.

Recall that, by Remark \ref{remfast}, we can put additional conditions of the form $m_n>M_n(m_1,\ldots,m_{n-1})$, and there still exists a sequence $(m_n)_n$ satisfying them all. 

\begin{proposition}[{\cite[Proposition 10.1 strengthened]{DiaGelRam:22a}}]\label{prothe:suspflow}
Assume the hypotheses of Section \ref{veryshortsec}. There exists a sufficiently fast-growing sequence of natural numbers $(m_n)_n$ such that for any $\sigma_{\cA}$-Bernoulli measure $\bp$ on $\cA^\bZ$, the associated sequence of measures $(\mu_n(\bp))_n$ converge in the weak$\ast$ topology to some limit measure $\mu_\infty(\bp)$.

Moreover, the sequence $(m_n)_n$ can be chosen sufficiently fast growing such that for every continuous function $\phi\colon\Sigma_N\times\bS^1\to\bR$ and $\varepsilon>0$, there exists $L_0=L_0(\phi,\varepsilon)\in\bN$ such that  for any $\sigma_{\cA}$-Bernoulli measure $\bp$ on $\cA^\bZ$ the measure preserving systems $(\fS_n,\Phi_n,\lambda_n(\bp))$, $n\in\bN$, satisfy the following. For every $\ell\ge L_0$ and $n\ge \ell+1$, there exists a subset $\fS_{n,\phi,\varepsilon}$ of $\fS_n$ such that $\lambda_n(\fS_{n,\phi,\varepsilon})>1-\varepsilon$ and for every $(\underline a,s)\in \fS_{n,\phi,\varepsilon}$ it holds
\begin{equation}\label{eq:expectedd}
	\left\lvert\frac{1}{\frakR_\ell}
		\sum_{k=0}^{\frakR_\ell-1}\psi_n(\Phi_n^k(\underline a,s))
		- \int\phi\,d\mu_\infty(\bp) \right\rvert
	<\varepsilon,
\end{equation}
where
\[
	\frakR_n
	\eqdef \int \underline R_n\,d\bp_n
	\spac{and}
	\psi_n\colon\fS_n\to\bR, \quad
	\psi_n(\underline a,s)\eqdef(\phi\circ H_n)(\underline a,s).
\]
\end{proposition}

The first fact in the following corollary, ergodicity, is an immediate consequence of the Gorodetski-Ilyashenko-Klep\-tsyn-Nalski argument (see \cite[Proposition 11.1]{DiaGelRam:22a}). By \cite[Corollary 1.2]{DiaFis:11} (see also \cite{CowYou:05}), in our setting, the entropy map is upper semi-continuous. Hence, the estimate of the entropy of $\mu_\infty(\bp)$ from below follows immediately from Proposition \ref{Corentest}.

\begin{corollary}\label{refcorollary}
The measure $\mu_\infty(\bp)$ is $F$-ergodic and satisfies
\[
	\chi(F,\mu_\infty(\bp))=0
	\,\text{ and }\,
	h(F,\mu_\infty(\bp))
	\ge \limsup_{n\to\infty}h(F,\mu_n(\bp))
	\ge \frac{e^{-L_1\lvert\alpha\rvert}}{\lVert\cW\rVert}  h(\sigma_\cA,\bp) .
\]
\end{corollary}

\subsubsection{Sketch of proof of Proposition \ref{prothe:suspflow}}

We start by informally describing the main idea of the proof of \cite[Proposition 10.1]{DiaGelRam:22a}. The use of suspension spaces $\fS_n$, that are topological extensions of the dynamics on the horseshoes $\Gamma_n$, on one hand allows us to properly define the measures. We  put some Bernoulli measure $\bp$ on $\cA^\bZ$, generating measures $\bp_n$ on the ground floor of the $n$th level horseshoe $\Gamma_n$, and look at its suspension measure $\lambda_n(\bp)$ and its projection $\mu_n(\bp)$ (compare again diagram \eqref{diagram2}). On the other hand, the fact that each suspension space has a designated ``ground floor'' and ``intermediate floors'' that are consequences of our recursive definition is very convenient. This recursive inherited internal floor structure is described in \cite[Section 9]{DiaGelRam:22a}. Indeed, a large part of each $n$th level horseshoe can be ``cut into pieces'' separated by ``intermediate floors'' whose dynamics are asymptotically governed by the distribution on the ``ground floor'' and such that each of those pieces is almost identical to the $k$th level horseshoe, $k\le n$. The term ``almost identical'' is made precise in \cite[Proposition 6.12]{DiaGelRam:22a}. 

As in the proof of \cite[Proposition 10.1]{DiaGelRam:22a}, to get the inequality \eqref{eq:expectedd} we need to estimate the Birkhoff sums of a given continuous potential along a ``typical'' long piece of trajectory in the $n$th level horseshoe. By the recursive structure of the horseshoes, such Birkhoff sum is comparable to a sum of Birkhoff sums over the pieces inside the $k$th level sub-horseshoes. As the base measure is Bernoulli, those Birkhoff sub-sums are independently and identically distributed and hence their sum can be estimated by a large deviation argument.

In \cite{DiaGelRam:22a}, the arguments are presented only in the particular case where the  Bernoulli measure  $\bp$ is the maximal entropy measure. For completeness, let us briefly describe where this special choice was used: 
\begin{itemize}
\item[(1)] \cite[Proposition 4.3]{DiaGelRam:22a} is based on Bernstein's inequality (see \cite[Lemma 4.4]{DiaGelRam:22a}) and makes use of the i.i.d. nature of Bernoulli measures. \item[(2)] \cite[Lemma 5.11]{DiaGelRam:22a} describes how the measure on the ground floor lifts to intermediate floors. The proof only uses the fact that the base measure is \emph{$\sigma_\cA$-invariant}.
\end{itemize}
Indeed, as a consequence, items (1) and (2) continue to be true when starting from \emph{any arbitrary} Bernoulli measure $\bp$ on $\cA^\bZ$. What we still need to show is that those arguments are uniform (the constants can be chosen independently from the base Bernoulli measure). We now address this specific point.

The first step of the proof in \cite[Proposition 10.1]{DiaGelRam:22a} is to fix some (arbitrary) dense sequence $(\phi_n)_n$ of continuous functions 
\begin{equation}\label{seqphin}
	\phi_n\colon\Sigma_N\times\bS^1\to\bR.
\end{equation}	 
The second essential step is, given this sequence $(\phi_k)_k$, to fix the sequence $(m_n)_n$.
Recalling \cite[Section 10.1]{DiaGelRam:22a}, the following three conditions have to be met:
\begin{enumerate}
\item[(I)] (controlled large deviation)	as in \cite[Proposition 4.3]{DiaGelRam:22a},
\item[(II)] (controlled distortion) as in \cite[Proposition 6.12]{DiaGelRam:22a},
\item[(III)] (recursive definition of tailing map) as in \cite[Theorem 7.3]{DiaGelRam:22a}.
\end{enumerate}
Item (II) is now repeated in the present paper as Lemma \ref{lemdistoprop}, it does not rely on the choice of Bernoulli measures. \cite[Theorem 7.3]{DiaGelRam:22a} is stated in this paper as Proposition \ref{pro:tailing}, it also does not depend on the choice Bernoulli measures. Only Condition (I) does depend on the Bernoulli measure $\bp$, this dependence is provided by \cite[Proposition 4.3]{DiaGelRam:22a}. Let us state its strengthened version that is independent of any Bernoulli measure and hence adapted to our needs. 

Given $\psi\colon\fS_{\cA,R}\to\bR$ continuous, let
\[
	\var_\cA(\Delta\psi)
	\eqdef \max_{a\in\cA}\max_{\underline b,\underline c\in[a]}
		\big(\Delta\psi(\underline b)-\Delta\psi(\underline c)\big),
\] 
where $\Delta\psi\colon\cA^\bZ\to\bR$ is defined by $\Delta\psi(\underline a)\eqdef\sum_{k=0}^{R(a_0)-1}\psi(\underline a,k)$.

\begin{lemma}[{\cite[Proposition 4.3 strengthened]{DiaGelRam:22a}}]
	Let $\psi\colon\fS_{\cA,R}\to\bR$ be a continuous potential. 
	For every $\varepsilon>0$ there exists $N_0=N_0(\psi,\varepsilon)\in\bN$ such that the following is true for every $m\ge N_0$. For every Bernoulli measure $\bp$ on $\cA^\bZ$ there exists a set $A\subset\cA^\bZ$ such that $\bp(A)>1-\varepsilon$ and for every $\underline a\in A$, $i=0,\ldots,m-1$, and $k\in\{1,\ldots,m\}$ it holds
\[\begin{split}
	\Big\lvert\sum_{j=i}^{i+k-1}\left(\underline R(\sigma_\cA^j(\underline a))
		-\int \underline R\,d\bp\right)\Big\rvert
	&<m\varepsilon,\\
	\Big\lvert\sum_{j=i}^{i+k-1}\left(\Delta\psi(\sigma_\cA^j(\underline a))
		-\int \Delta\psi\,d\bp\right)\Big\rvert
	&<m(2\var_\cA(\Delta\psi)+\varepsilon).
\end{split}\]
\end{lemma} 

\begin{proof}
Following \cite[Proof of Proposition 4.3]{DiaGelRam:22a}, it suffices to observe that $m=m(\bp)$  provided by the Bernstein inequality in \cite[Lemma 4.4]{DiaGelRam:22a} depends continuously on $\bp$. Then take into account that the set of Bernoulli measures on $\cA^\bZ$ is compact. 
\end{proof}

The above implies that the proof of \cite[Proposition 10.1]{DiaGelRam:22a} goes through choosing a sequence $(m_n)_n$ such that its assertion holds simultaneously for all Bernoulli measures. 
This ends our sketch of the proof of Proposition \ref{prothe:suspflow}.
\qed

\begin{lemma}
The diagram in Figure \ref{diagram3bright} commutes.
\end{lemma}

\begin{proof}
	The weak$\ast$ convergence $\mu_n(\bp)\to\mu_\infty(\bp)$ implies the weak$\ast$ convergence  $\pi_\ast(\mu_n(\bp))\to\pi_\ast(\mu_\infty(\bp))$. At the same time, $\pi_\ast(\mu_n(\bp))=\nu_n(\bp)$ converge $\bar f$ (and hence weak$\ast$) to $\nu_\infty(\bp)$. Hence, the diagram commutes. 
\end{proof}

We close this section with an auxiliary result about the Wasserstein distance of the measures of the sequence $(\mu_n(\bp))_n$. We will use it in the proof of Theorem \ref{Bthm:circleb}.
 
\begin{proposition}[Estimates of the Wasserstein distance]
Assume the hypotheses of Section \ref{veryshortsec}. Then for every $\varepsilon>0$, there exists a sufficiently fast-growing sequence of natural numbers $(m_n)_n$ such that the associated sequence of measures $(\mu_n(\bp))_n$ satisfy  for every $n\ge0$
\[
	W(\mu_0(\bp), \mu_n(\bp)) \leq \varepsilon + 8 L_1 \diam(\Sigma_N\times\bS^1)\cdot|\chi(F,\theta)|.
\]  
\end{proposition}

\begin{proof}
It suffices to see that the initial number $m_1$ needs to be chosen sufficiently large. The choice of the remaining sequence $(m_n)_{n\ge2}$ is unaltered. Fix $\varepsilon>0$ and assume that $m_1\ge \max\{N_1,N_2\}$, where $N_1=N_1(\varepsilon)$ and $N_2=N_2(\varepsilon)$ were chosen  as in Lemma \ref{lemdistoprop} and Lemma \ref{lem:sbe}, respectively.

Recall that $\alpha=\chi(F,\theta)<0$. Recall that, by Remark \ref{remWass},
\[
		W(\mu_0(\bp), \mu_n(\bp))
		= \sup\Big\{\Big|\int\phi\,d\mu_0(\bp)-\int\phi\,d\mu_n(\bp)\Big|\colon\phi\in\Lip_0(1)\Big\}.
\]
Fix $\phi \in \Lip_0(1)$. Consider the function $\Phi\colon\cW^{m_1} \to \bR$ defined for $v=w_1\ldots w_{m_1}$ as 
\begin{equation}\label{defPhi}
	\Phi(v) 
	=
		\max_{(\xi,x),(\eta,y)\in[w_1\ldots w_{m_1}]\times J}
		\sum_{i=0}^{|w_1\ldots w_{m_1}|-1} \phi(F^i(\xi,x)).
\end{equation}
Recalling that $\phi$ is 1-Lipschitz, by Lemma \ref{lem:sbe}, for any $(\eta,y)\in[v]\times J$ we have
\begin{equation} \label{eqn:legs}
	\Phi(v)- \varepsilon |v| 
	\leq \sum_{i=0}^{|v|-1} \phi(F^i(\eta,y)) 
	\le \Phi(v),
\end{equation}

Let $\psi_0 = \phi\circ H_0$ and $\psi_n = \phi\circ H_n$ be the retracted copies of $\phi$ on the suspension spaces $\fS_0$ and $\fS_n$, respectively. Let us now estimate 
\begin{enumerate}
\item[(a)] $\int \phi \,d\mu_0(\bp) = \int \psi_0 \,d\lambda(\bp)$,
\item[(b)] $\int \phi\, d\mu_n(\bp) = \int \psi_n \,d\lambda_n(\bp_n)$.
\end{enumerate}
in terms of the Birkhoff sums for generic sequences for $\lambda(\bp)$ and $\lambda_n(\bp_n)$, respectively (note again that both measures are ergodic). 

\smallskip\noindent\textbf{The integral (a).}
Recall first Remark \ref{remabstract}: we identify the family of words $\cW$ with a set of letters in an abstract alphabet $\cA$ that has the same cardinality $\card\cA=\card\cW$. Hence, the map $\Phi$ defined in \eqref{defPhi} can be seen as a function on $\cA^{m_1}$.
In \eqref{eqn:legs} we estimate from below and above the Birkhoff sum of $\psi_0$ along a piece of trajectory starting at the ``ground floor'' of the suspension space
\[
	\cA^\bZ\times \{0\}
\]
and ending at the $m_1$-st return to this ground floor, passaged through the sequence of returns codified by the word $v=w_1\ldots w_{m_1}\in \cW^{m_1}$ (note that $H_0$ projects the ground floor into $\Sigma_N\times J$, that is, the $\bS^1$-coordinate at those times is indeed in $J$ and all distortion control arguments from CIFSs apply). Any $\lambda(\bp)$-generic trajectory can be divided into such blocks of returns and the probability of choosing a particular sequence $v$ of blocks is $\bp([v])=\widetilde\bp([v])$. Indeed, note that we are at the beginning of our repeat-and-tailing scheme and have not yet added any tail; moreover, the roof function $R(\cdot)=|\varrho_0|$ is constant and hence it holds $\widetilde\bp=\bp$ (recall Lemma \ref{susLBpre}). Thus, using \eqref{eqn:legs}, the $\lambda(\bp)$-expected value can be estimated by
\begin{equation}\label{somenumber}
\frac{1}{m_1\lVert\cW\rVert}\sum_{v\in \cA^{m_1}} \widetilde\bp([v]) \Phi(v) - \varepsilon 
\leq \int\psi_0\,d\lambda(\bp)
\leq \frac{1}{m_1\lVert\cW\rVert}\sum_{v\in \cA^{m_1}} \widetilde\bp([v]) \Phi(v).
\end{equation}

\smallskip\noindent\textbf{The integral (b).}
The formula \eqref{eqn:legs} also estimates the Birkhoff sum of $\psi_n$ along a piece of trajectory starting at the ``intermediate floor of level 1'' and ending at the $m_1$-st roof (the beginning of a tail word $\bt_1(\cdot)$). Observe that, as before, at those times the projection of the trajectory by $H_n$ is in $\Sigma_N\times J$. The probability of choosing a particular word $v$ is given by 
\[
(\underline\cS_{n,0})_\ast \widetilde{\bp_n} ([v]).
\]
We could proceed almost as before. However, we need to take into account that the $\lambda_n(\bp_n)$-generic trajectory contains not only those blocks codified by $v$, but also the ``tail blocks'' between them, and we need to take them into account. By Corollary \ref{cor:maxmin} together with our setting (iv) in Section \ref{veryshortsec}, for all $n$
\[	1
	\le \frac{\max|\varrho_n|}{\min|\varrho_n|}
	\le 1+4K,\quad\text{ where }\quad
	K\eqdef L_1|\alpha|.
\]
Recall that, by Corollary \ref{cor:maxmin}, $4K/(1+4K)$ estimates the maximal total relative length of added tails (at any level); we estimate the contribution of $\phi\in\Lip_0(1)$ along these tails simply by $\min\phi\ge0$ and $\max\phi$ from below and above, respectively:
\[\begin{split}
	\Big(1-&\frac{4K}{1+4K}\Big) \frac{1}{m_1\lVert\cW\rVert}\sum_{v\in \cA^{m_1}} (\underline\cS_{n,0})_\ast \widetilde{\bp_n}([v]) \Phi(v) 
	- \varepsilon 
	+  \frac{4K}{1+4K} \min \phi\\
&\leq 
 \int\psi_n\,d\lambda_n(\bp_n)\\
&\leq\Big(1-\frac{4K}{1+4K}\Big)\frac{1}{m_1\lVert\cW\rVert}  \sum_{v\in \cA^{m_1}} (\underline\cS_{n,0})_\ast \widetilde{\bp_n}([v]) \Phi(v) 
	+ \frac{4K}{1+4K}\max \phi
\end{split}\]
Together with Corollary \ref{corclavonhieraus}, we get
\begin{equation}\label{somenumber2}\begin{split}
	\frac{1}{1+4K}&(1-4K)\frac{1}{m_1\lVert\cW\rVert} \sum_{v\in \cA^{m_1}} \widetilde\bp([v]) \Phi(v) 
	- \varepsilon \\
&\leq 
	  \int\psi_n\,d\lambda_n(\bp_n)\\
&\leq  \frac{1}{1+4K}(1+4K)\frac{1}{m_1\lVert\cW\rVert} \sum_{v\in \cA^{m_1}} \widetilde\bp([v]) \Phi(v) 
	+ \frac{4K}{1+4K} \max \phi	.
\end{split}\end{equation}
This completes the estimate of the integral in (b).

We are now ready to prove the proposition. Note  that
\[
	\min\phi
	\le \frac{1}{m_1\lVert\cW\rVert} \sum_{v\in \cA^{m_1}} \widetilde\bp([v]) \Phi(v) 
	\le\max\phi.
\]	
As
\[
	 \left|\int  \phi \,d\mu_0(\bp) -\int  \phi \,d\mu_n(\bp) \right|
	= \left|\int  \psi_0 \,d\lambda(\bp) -\int  \psi_n \,d\lambda_n(\bp_n) \right|,
\]
by \eqref{somenumber} and \eqref{somenumber2}, 
\[
	- \varepsilon
	- \frac{-4K}{1+4K}2\max\phi
	\le \int  \phi \,d\mu_n(\bp) -\int  \phi \,d\mu_0(\bp) 
	\le 
	 \frac{4K}{1+4K} \max \phi
	+\varepsilon.
\]	
Note again that $\phi\in\Lip_0(1)$ implies $0\le\min\phi\le\max\phi\le \diam(\Sigma_N\times\bS^1)$. Hence,
\[
	 \left|\int  \phi \,d\mu_n(\bp) -\int  \phi \,d\mu_0(\bp) \right|
	\le 8K \diam(\Sigma_N\times\bS^1)
	+\varepsilon.
\]
Finally recall that, in the setting in Section \ref{veryshortsec}, $K=L_1|\chi(F,\theta)|$.
This ends the proof. 
\end{proof}

\subsection{Weak* continuity of the limit measures}\label{sec52b}

In this section, we describe the objects from Section \ref{seccocyclediffeo} using the language of a cascade of substitutions in Section \ref{secabscasc}. In particular, we consider the  target space $\cB=\{1,\ldots,N\}$ and $\cB^\ast=\Sigma_N^\ast$ with the substitutions $\varrho_n\colon\cA_n\to\Sigma_N^\ast$ defined as follows. Given $\cW\subset\Sigma_N^\ast$ and the associated abstract alphabet $\cA$ in Section \ref{veryshortsec}, let
\[
	\varrho_0\colon\cA\to\Sigma_N^\ast,\quad
	\varrho_0(a_k)\eqdef w_k \in\cW.
\]
Recalling the definition of the tailing map $\bt_n$ in \eqref{tailingmap}, we define iteratively for every $n \in \bN$ the maps
\begin{equation}\label{defsubs}\begin{split}
	&\varrho_n\colon\cA_n\to\Sigma_N^\ast,\\
	&\varrho_n(a^{(n)})
	\eqdef \varrho_{n-1}(a_0^{(n-1)})\ldots\varrho_{n-1}(a_{m_n-1}^{(n-1)})
			\bt_n\big(\varrho_{n-1}(a_0^{(n-1)})\ldots\varrho_{n-1}(a_{m_n-1}^{(n-1)})\big).
\end{split}\end{equation}
Repeated application of Proposition \ref{pro:tailing} guarantees that this cascade satisfies Assumption \ref{assumption1} taking $K=L_1\lvert\alpha\rvert$.

In what is below, we vary the Bernoulli measure $\bp\in\cM_{\rm B}(\cA^\bZ,\sigma_\cA)$ and study the  corresponding limit measures $\mu_\infty(\bp)$ as provided by Proposition \ref{prothe:suspflow}.

\begin{proposition}\label{prop:geometriczero}
	Assume the hypotheses of Section \ref{veryshortsec} and let $(m_n)_n$ be as in Proposition \ref{prothe:suspflow}. Then the map 
\[
	\cM_{\rm B}(\cA^\bZ,\sigma_\cA)\ni\bp\mapsto\mu_\infty(\bp)\in\cM_{\rm erg}(\Sigma_N\times\bS^1,F)
\]	 
is continuous in the weak$\ast$ topology. 
\end{proposition}

\begin{proof}
 The following fact is an immediate consequence of our construction of each measure $\mu_n(\bp)$ (being a factor of a discrete suspension, compare again Figure \ref{diagram2}).

\begin{claim}\label{claimcorcontweakast}
	The map $\bp\mapsto\mu_n(\bp)$ is continuous in the weak$\ast$ topology.
\end{claim}

We are now looking at $\mu_n(\cdot)$ as a function
\begin{equation}\label{Yn}
	\mu_n\colon\cM_{\rm B}(\cA^\bZ,\sigma_\cA)\to\cM_{\rm erg}(\Sigma_N\times\bS^1,F).
\end{equation}
By Claim \ref{claimcorcontweakast}, each $\mu_n(\cdot)$ is continuous. By Proposition \ref{prothe:suspflow}, the sequence $(\mu_n(\bp))_n$ is pointwise converging  (in the weak$\ast$ topology) and we can define
\[
	\mu_\infty\colon\cM_{\rm B}(\cA^\bZ,\sigma_\cA)\to\cM_{\rm erg}(\Sigma_N\times\bS^1,F),\quad
	\mu_\infty(\bp)
	\eqdef\lim_{n\to\infty}\mu_n(\bp)= \mu_\infty(\bp).
\]

Let us check that the family \eqref{Yn} is equicontinuous (in the weak$\ast$ topologies). 

\begin{lemma}\label{lem:equicont}
	For every $\varepsilon>0$ there is $n_0\in\bN$ such that for every $n\ge n_0$ and $\bp\in\cM_{\rm B}(\cA^\bZ,\sigma_\cA)$, we have
\[
	W(\mu_n(\bp),\mu_\infty(\bp))
	\le \varepsilon.
\]	 
\end{lemma}

\begin{proof}
Recall that, by \eqref{wasssb}, to determine the Wasserstein distance, it is enough to restrict ourselves to the space $\Lip_0(1) \subset \Lip(1)$ of non-negative $1$-Lipschitz functions whose values are bounded by $\diam(\Sigma_N\times\bS^1)$. As the sequence $(\phi_n)_n$ of continuous functions in \eqref{seqphin} fixed in the proof of Proposition \ref{prothe:suspflow} is dense in the space of all continuous functions, for every $\varepsilon>0$, it contains a finite subcollection $\{\phi_1,\ldots,\phi_{i(\varepsilon)}\}$ such that any function in $\Lip_0(1)$ is $\varepsilon$-close to one function in this collection. Hence, for any $\bp\in\cM_{\rm B}(\cA^\bZ,\sigma_\cA)$, 
\[\begin{split}
	W(\mu_n(\bp),\mu_\infty(\bp))
	&= \sup_{\phi\in\Lip_0(1)}\Big|\int\phi\,d\mu_n(\bp)-\int\phi\,d\mu_\infty(\bp)\Big|\\
	&\le 2\varepsilon
		+\max_{k=1,\ldots,i(\varepsilon)}
			\Big|\int\phi_k\,d\mu_n(\bp)-\int\phi_k\,d\mu_\infty(\bp)\Big|.
\end{split}\]	
Recall the topological factor $H_n$ (see Figure \ref{diagram2}).
By Proposition \ref{prothe:suspflow}, if $n\in\bN$ is large enough, then for every $k=1,\ldots,i(\varepsilon)$ there is some subset $\fT_n=\fT_{n,\phi_k,\varepsilon}\subset \Sigma_N\times\bS^1$, $\fT_n=H_n(\fS_{n,\phi_k,\varepsilon})$, with $\mu_n(\fT_n)>1-\varepsilon$ such that
\[\begin{split}
	\int&\phi_k\,d\mu_n(\bp)=\\
	{\tiny\text{($F$-invariance of $\mu_n$)}}\,
	&=\int\frac{1}{\frakR_\ell}\sum_{k=0}^{\frakR_\ell-1}\phi_k\circ F^k\,d\mu_n(\bp)\\
	&=\int_{\fT_n}\frac{1}{\frakR_\ell}\sum_{k=0}^{\frakR_\ell-1}\phi_k\circ F^k\,d\mu_n(\bp)
	+\int_{\fT_n^c}\frac{1}{\frakR_\ell}\sum_{k=0}^{\frakR_\ell-1}\phi_k\circ F^k\,d\mu_n(\bp)\\
	{\tiny\text{(by \eqref{eq:expectedd} and $\mu_n(\fT_n^c)<\varepsilon$)}}\,
	&\le \int_{\fT_n}\Big(\int\phi_k\,d\mu_\infty(\bp)+\varepsilon\Big)\,d\mu_n(\bp)
	+\varepsilon\max|\phi_k|\\
	&\le \int\phi_k\,d\mu_\infty(\bp)	+\varepsilon\mu_n(\fT_n)+ \varepsilon\max|\phi_k| ,
\end{split}\]	
together with the analogous lower bound. 
Hence, together with the fact that $\phi_k$ is $\varepsilon$-close to some function in $\Lip_0(1)$ and the fact that $\mu_n$ is a probability measure, we get
\[
	\left|\int\phi_k\,d\mu_n(\bp)-\int\phi_k\,d\mu_\infty(\bp)\right|
	\le 3\varepsilon+ \varepsilon\diam(\Sigma_N\times\bS^1).
\]
Hence, we get
\[
	W(\mu_n(\bp),\mu_\infty(\bp))
	\le \varepsilon(3+\diam(\Sigma_N\times\bS^1)).	
\]
Note again that the above estimates does not depend on $\bp$. This proves the lemma. 
\end{proof}

To finish the proof our proposition, just observe that Lemma \ref{lem:equicont}, together with continuity of $\bp\mapsto\mu_n(\bp)$ and pointwise convergence $\mu_n(\bp)\to\mu_\infty(\bp)$ implies continuity of $\bp\mapsto\mu_\infty(\bp)$, all in the weak$\ast$ topology. 
\end{proof}

\section{Proofs of Theorems \ref{Bthm:circle} and \ref{Bthm:circleb}}\label{sec7}

We continue to assume the setting collected in Section \ref{veryshortsec} and the notation therein. Throughout this section, we assume that $F\in\mathrm{SP}^1_{\rm shyp}(\Sigma_N\times\bS^1)$ and that $J\subset\bS^1$ is some associated blending interval.

\subsection{Proof of Theorem \ref{Bthm:circle}}\label{secthmcircle}

By \cite[Theorem A]{DiaGelRam:22a}, there exists an $F$-ergodic measure $\theta$ with fiber Lyapunov exponent $\alpha$ negative and arbitrarily close to $0$ and entropy arbitrarily close to $h_0(F)$. To any such measure $\theta$ we apply the construction summarized in Section \ref{veryshortsec} and afterwards apply Proposition \ref{prop:geometriczero}. Thus, letting $\cM_\varepsilon\eqdef\{\mu_\infty(\bp)\colon\bp\in\cM_{\rm B}(\cA^\bZ,\sigma_\cA)\}$, this  is a weak$\ast$ path-connected set of $F$-ergodic measures with zero fiber Lyapunov exponent. Moreover, by Remark \ref{remsimples}, the set $\cM_\varepsilon$ contains a measure with zero entropy and,  Corollary \ref{refcorollary}, a measure whose entropy is bounded from below by $h_0(F)-\varepsilon$, for some $\varepsilon>0$. 

It follows from Proposition \ref{pro:tailing} that the substitutions $\varrho_n \colon\cA_n\to\Sigma_N^\ast$, given in \eqref{defsubs}, satisfy Assumption \ref{assumption1}, with $K=L_1\lvert\alpha\rvert$. Hence, we can invoke all tools from Section \ref{secabscasc}.  By Theorem \ref{theoprop:pathfinal}, $\nu_\infty(\cdot)$ is a well defined and $\bar f$-continuous function from $\cM_{\rm B}(\cA^\bZ,\sigma_\cA)$ into $\cM_{\rm erg}(\Sigma_N,\sigma)$. Moreover, $\nu_\infty(\cM_{\rm B}(\cA^\bZ,\sigma_\cA))$ is an $\bar f$-connected set. All the measures in this set are LB and their entropies vary continuously. In particular, the set of entropies of those measures is a closed interval.

By Lemma \ref{entpres}, $\pi_\ast$ preserves the entropy. Hence, the remaining assertions of the theorem follow from the fact that the diagram in Figure \ref{finalfigure} commutes, which is a consequence of Claim \ref{claimfinal}.
\qed

\begin{figure}[h]
	\hspace{0.2cm}	\xymatrix{
	\big(\cM_{\rm B}(\cA^\bZ,\sigma_\cA), \text{weak$\ast$ topology}\big) \ar[r]^{\mu_\infty\,\,\,\,\,} \ar[rd]_{\nu_\infty}
		&\big(\mu_\infty(\cM_{\rm B}(\cA^\bZ,\sigma_\cA)),\text{weak$\ast$ topology}\big)\ar[d]^{\pi_\ast}\\
		&\big(\nu_\infty(\cM_{\rm B}(\cA^\bZ,\sigma_\cA)),\bar f\text{-topology}\big)
	}
\caption{Relations between limit measures: proof of Theorem \ref{Bthm:circle}}
\label{finalfigure}	
\end{figure}

\subsection{Proof of Theorem \ref{Bthm:circleb}}

We start our proof by collecting the following results.

\begin{lemma}[{\cite[Theorem 1]{DiaGelRam:17}}]\label{lemthetainfty}
For every $\theta^\infty\in\cM_{\rm erg,0}(\Sigma_N\times\bS^1,F)$, there exists a sequence $(\theta^\ell)_\ell\subset \cM_{\rm erg}(\Sigma_N\times\bS^1,F)$ satisfying $\alpha_\ell\eqdef\chi(F,\theta^\ell)<0$ for every $\ell$ such that $\theta^\ell\to\theta^\infty$ in Wasserstein distance (and hence, in particular, $\lim_\ell\alpha_\ell=0$) and in entropy  . 
\end{lemma}

\begin{lemma}\label{lemchoicebpell}
For every $\varepsilon>0$ and $\theta\in\cM_{\rm erg}(\Sigma_N\times\bS^1,F)$ with $\alpha=\chi(F,\theta)<0$,  there is a disjoint finite collection of words $\cW\subset\Sigma_N^\ast$ of equal lengths defining a CIFS on $J$, an associated abstract alphabet $\cA$, and a Bernoulli vector $\bp$ on $\cA^\bZ$ such that $\mu_0(\bp)$ is $\varepsilon$-close to $\theta$ in Wasserstein distance and in entropy.
\end{lemma}

\begin{proof}
The existence of $\cW$ follows from Proposition \ref{proteo:existenceCIFS}. Letting $\cA$ be the associated abstract alphabet, $\card\cA=\card\cW$, and $\bp$ be the maximal entropy Bernoulli measure on $\cA^\bZ$. Together with Proposition \ref{Corentest} we get
\[
	h(F,\mu_n(\bp))
	\ge \frac{e^{-L_1|\alpha|}}{\lVert\cW\rVert}h(\sigma_\cA,\bp) 
	= \frac{e^{-L_1|\alpha|}}{\lVert\cW\rVert}\log\card\cA
	\ge e^{-L_1|\alpha|}(h(F,\theta)-\varepsilon_H),
\]
for every $n\ge0$. This implies the assertion.
\end{proof}

To continue the proof of Theorem \ref{Bthm:circleb}, fix $\theta^\infty\in\cM_{\rm erg,0}(\Sigma_N\times\bS^1,F)$. Consider a sequence $(\theta^\ell)_\ell$ provided by Lemma \ref{lemthetainfty} such that
\[
	W(\theta^\ell,\theta^\infty)<2^{-\ell},\quad
	|h(F,\theta^\ell)-h(F,\theta^\infty)|<2^{-\ell},\quad
	2^{-\ell}<\alpha_\ell\eqdef\chi(F,\theta^\ell)<0.
\]
Without loss of generality, as $\pi_\ast\colon\cM_{\rm erg}(\Sigma_N\times\bS^1,F)\to\cM_{\rm erg}(\bS^1,\sigma)$ is continuous and preserves entropy, we can also assume that
\begin{equation}\label{eq1}
	W(\pi_\ast\theta^\ell,\pi_\ast\theta^\infty)<2^{-\ell},\quad
	|h(\sigma,\pi_\ast\theta^\ell)-h(\sigma,\pi_\ast\theta^\infty)|<2^{-\ell}.
\end{equation}

To every ergodic measure $\theta^\ell$, we now invoke the setting (i)--(vi) in Section \ref{veryshortsec}. In particular, for every index $\ell$, we obtain disjoint finite collection of words of equal lengths $\cW^\ell$ together with corresponding cascades of abstract alphabets $\cA^\ell_n$, suspension spaces $\cS^\ell_n$, horseshoes $\Gamma^\ell_n$ and $F$-ergodic measures $\mu^\ell_n(\bp^\ell)$, and $\sigma$-ergodic measure $\nu^\ell_n(\bp^\ell)$. Note that in item (iv) for every $\ell$ we take $K_\ell\eqdef L_1|\alpha_\ell|$.
For each $\ell\in\bN$, let $\mu_0^\ell(\bp^\ell)$ be the ergodic measure as provided by Lemma \ref{lemchoicebpell} applied to $\varepsilon_\ell$ sufficiently small and $\theta^\ell$ with fiber Lyapunov exponent $\alpha_\ell<0$ such that
\[
	W(\mu_0^\ell(\bp^\ell),\theta^\ell)<2^{-\ell},\quad
	|h(F,\mu_0^\ell(\bp^\ell))-h(F,\theta^\ell)|<2^{-\ell}.
\]
As before, and also using that $\pi_\ast\mu_0^\ell(\bp^\ell)=\nu_0^\ell(\bp^\ell)$, we can assume that also
\begin{equation}\label{eq2}
	W(\nu_0^\ell(\bp^\ell),\pi_\ast\theta^\ell)<2^{-\ell},\quad
	|h(\sigma,\nu_0^\ell(\bp^\ell))-h(\sigma,\pi_\ast\theta^\ell)|<2^{-\ell}.
\end{equation}
By Proposition \ref{prokickoff}, for all $n$
\[
	\bar f\big(\nu_n^\ell(\bp^\ell),\nu_0^\ell(\bp^\ell)\big)
	\le 6L_1|\alpha_\ell|+ 8L_1^2|\alpha_\ell|^2.
\]
Hence,
\begin{equation}\label{eq3}
	\bar f\big(\nu_\infty^\ell(\bp^\ell),\nu_0^\ell(\bp^\ell)\big)
	\le 6L_1|\alpha_\ell|+ 8L_1^2|\alpha_\ell|^2.
\end{equation}

Finally, recall that $\bar f$-convergence implies convergence in the weak$\ast$ topology and in entropy, see Remark \ref{remfbars}. Hence, from \eqref{eq1}--\eqref{eq3} we obtain
\[
	\lim_{\ell\to\infty}\nu_\infty^\ell(\bp^\ell) 
	= \lim_{\ell\to\infty}\nu_0^\ell(\bp^\ell) 
	= \lim_{\ell\to\infty}\pi_\ast\theta^\ell 
	= \pi_\ast\theta_\infty,
\]
where convergence is in the weak$\ast$ topology and in entropy.
\qed

\section{Implications for matrix cocycles and proofs of Theorems \ref{theMain} and \ref{theMainb}}	\label{ImpMatCoc}

Let $\bA=\{A_1,\ldots,A_N\}\in\mathrm{SL}(2, \mathbb{R})^N$, $N\ge2$, be a finite collection of matrizes. Consider the associated step skew product $F_\bA$ defined as in \eqref{eq:sp} with fiber maps given by \eqref{neq:defsteskecoc}.
The top Lyapunov exponents of the cocycle $\bA$ and the fiber Lyapunov exponent of the skew product 
$F_\bA$ are related as explained in the following immediate consequence of \cite[Theorem 11.1]{DiaGelRam:19} and of Lemma \ref{entpres}.

\begin{lemma}\label{lemcontin}
For every ergodic measure $\nu\in\cM_{\rm erg}(\Sigma_N,\sigma)$ and for every $F_\bA$-ergodic measure $\mu$ such that $\nu=\pi_\ast \mu$, we have
\[
\chi(F_\bA,\mu) \in\{  -2 \lambda_1(\bA,\nu),2 \lambda_1(\bA,\nu)\}
\quad\text{ and }\quad
h(F_\bA,\mu)
	= h(\sigma,\nu).
\]
\end{lemma}

We refrain from repeating the full definition of the set $\fE_{N, \rm shyp}$ (see \cite[Section 11.7]{DiaGelRam:19}); a rough presentation goes as follows: a cocycle $\bA$ belongs to $\fE_{N, \rm shyp} $ if it has 
\begin{enumerate}
\item Some hyperbolicity: The semi-group generated by $\bA$ contains a hyperbolic element $A$ and let
\[
	B_A^-\eqdef\{v\in\bP^1\colon|f_A'(v)|<1\},\quad
	B_A^+\eqdef\{v\in\bP^1\colon|f_A'(v)|>1\},
\]
where $f_A$ is defined as in \eqref{neq:defsteskecoc};
\item Transitions in finite time: There is $M\in\bN$ such that for every $v\in\bP^1$ there are sequences $\xi^+,\eta^+\in\Sigma_N^+$ such that $f_{\xi_{s-1}}\circ\cdots\circ f_{\xi_0}(v)\in B_A^+$ and $f_{\eta_{r-1}}\circ\cdots\circ f_{\eta_0}(v)\in B_A^-$ for some $s,r\le M$.
\end{enumerate}
Two crucial facts are that $\fE_{N, \rm shyp}$ is open and dense in $\fE_N$ and that for every $\bA\in\fE_{N, \rm shyp}$ the induced skew product $F_\bA$ is in $\mathrm{SP}^1_{\rm shyp}(\Sigma_N\times\bP^1)$ (see \cite[Proposition 11.23]{DiaGelRam:19}). 

Let us now argue that Theorem \ref{theMain} is indeed a consequence of Theorem \ref{Bthm:circle}. 

\subsection{Proof of Theorem \ref{theMain}}

Let $\bA$ be a matrix cocycle in $\mathfrak{E}_{N,\rm shyp}$. Thus, the associated skew product map $F_\bA$ is in $\mathrm{SP}^1_{\rm shyp}(\Sigma_N\times\bS^1)$, and hence we can apply Theorem \ref{Bthm:circle}.
As shown in Section \ref{secthmcircle} (proof of Theorem \ref{Bthm:circle}), the diagram in Figure \ref{finalfigure} commutes and all maps in it are continuous in the indicated topologies.

	By Theorem \ref{Bthm:circle}, the set $\mu_\infty(\cM_{\rm B}(\cA^\bZ,\sigma_\cA))$ is a path-connected family of ergodic measures with zero fiber Lyapunov exponent. Moreover, their entropies contain the interval $[0,h_0(F)-\varepsilon]$, where $\varepsilon$ comes from Theorem \ref{Bthm:circle}. Lemma \ref{entpres} implies that $h_0(F)=h_0(\bA)$. Moreover, the image of $\mu_\infty(\cM_{\rm B}(\cA^\bZ,\sigma_\cA))$ under $\pi_\ast$ is $\bar f$-path connected and consists only of LB measures. By Lemma \ref{lemcontin} all measures in $\pi_\ast(\mu_\infty(\cM_{\rm B}(\cA^\bZ,\sigma_\cA)))$ have zero top Lyapunov exponent for $\bA$. 
\qed

\subsection{Proof of Theorem \ref{theMainb}}

Let $\bA$ be a matrix cocycle in $\mathfrak{E}_{N,\rm shyp}$. Thus, the associated skew product map $F_\bA$ is in $\mathrm{SP}^1_{\rm shyp}(\Sigma_N\times\bS^1)$, and hence we can apply Theorem \ref{Bthm:circleb}.

Let $\nu^+\in\cM_{\rm erg,0}(\Sigma_N^+,\sigma^+)$. For $\nu^+$-almost every $\xi^+$, for every $v\in\bP^1$, we have $\chi^+(\xi^+,v)=0$ (see \cite[Proposition 11.5]{DiaGelRam:22a}). Let $\nu\in\cM_{\rm erg}(\Sigma_N,\sigma)$ the natural extension of $\nu^+$. This immediately implies that every $\mu\in\cM_{\rm erg}(\Sigma_N\times\bS^1,F)$ such that $\pi_\ast\mu=\nu$ satisfies $\chi(F_\bA,\mu)=0$.
 The assertion now follows from Theorem \ref{Bthm:circleb}.
\qed

\bibliographystyle{alpha}

\end{document}